\documentclass[12pt]{amsart}

\usepackage{amscd}
\usepackage{amsmath, amssymb}
\usepackage{graphicx,color}
\usepackage{amsfonts}
\newcommand{\de}{\partial}

\newcommand{\dbar}{\overline{\partial}}

\newcommand{\ddb}{\partial \ov{\partial}}

\newcommand{\ddbar}{\sqrt{-1} \partial \overline{\partial}}

\newcommand{\ov}[1]{\overline{#1}}
\newcommand{\mn}{\sqrt{-1}}

\newcommand{\ti}[1]{\tilde{#1}}
\newcommand{\vp}{\varphi}
\newcommand{\ve}{\epsilon}

\def\a{{\alpha}}

\renewcommand{\leq}{\leqslant}
\renewcommand{\geq}{\geqslant}

\newcommand{\be}{\begin{equation}}
\newcommand{\ee}{\end{equation}}

\def\Re{\mathrm{Re}}
\def\Im{\mathrm{Im}}
\def\Ren{\Re\left(e^{-\sqrt{-1}\hat\theta}\Omega_\varphi^n \right)}
\def\Renn{\Re\left(e^{-\sqrt{-1}\hat\theta}\Omega_\varphi^{n-1} \right)}

\def\Imn{\Im\left(e^{-\sqrt{-1}\hat\theta}\Omega_\varphi^n \right)}
\def\Imnn{\Im\left(e^{-\sqrt{-1}\hat\theta}\Omega_\varphi^{n-1} \right)}
\begin{document}
\newtheorem{claim}{Claim}
\newtheorem{theorem}{Theorem}[section]
\newtheorem{lemma}[theorem]{Lemma}
\newtheorem{corollary}[theorem]{Corollary}
\newtheorem{proposition}[theorem]{Proposition}
\newtheorem{question}{question}[section]
\theoremstyle{definition}
\newtheorem{definition}[theorem]{Definition}
\newtheorem{remark}[theorem]{Remark}

\numberwithin{equation}{section}

\title[The space of almost calibrated $(1,1)$ forms ]{The space of almost calibrated $(1,1)$ forms on a compact K\"ahler manifold}

\author[J. Chu]{Jianchun Chu}
\address{Department of Mathematics, Northwestern University, 2033 Sheridan Road, Evanston, IL 60208}
\email{jianchun@math.northwestern.edu}
\author[T. C. Collins]{Tristan C. Collins}
\address{Department of Mathematics, Massachusetts Institute of Technology, 77 Massachusetts Avenue, Cambridge, MA 02139}
 \thanks{T.C.C is supported in part by NSF grant DMS-1810924, NSF CAREER grant DMS-1944952 and an Alfred P. Sloan Fellowship. }
\email{tristanc@mit.edu}
\author[M.-C. Lee]{Man-Chun Lee}
\address{Department of Mathematics, Northwestern University, 2033 Sheridan Road, Evanston, IL 60208}
\email{mclee@math.northwestern.edu}
\thanks{M.-C. Lee is supported in part by NSF grant 1709894. }

\begin{abstract}
The space $\mathcal{H}$ of ``almost calibrated" $(1,1)$ forms on a compact K\"ahler manifold plays an important role in the study of the deformed Hermitian-Yang-Mills equation of mirror symmetry as emphasized by recent work of the second author and Yau \cite{CY18}, and is related by mirror symmetry to the space of positive Lagrangians studied by Solomon.  This paper initiates the study of the geometry of $\mathcal{H}$.  We show that $\mathcal{H}$ is an infinite dimensional Riemannian manifold with non-positive sectional curvature.  In the hypercritical phase case we show that $\mathcal{H}$ has a well-defined metric structure, and that its completion is a ${\rm CAT}(0)$ geodesic metric space, and hence has an intrinsically defined ideal boundary.  Finally, we show that in the hypercritical phase case $\mathcal{H}$ admits $C^{1,1}$ geodesics, improving a result of the second author and Yau \cite{CY18}.  Using results of Darvas-Lempert \cite{DL12} we show that this result is sharp.
\end{abstract}

\maketitle

\section{Introduction}

Let $(X,\omega)$ be a compact $n$-dimensional K\"ahler manifold and $[\alpha]$ be a class in $H^{1,1}(X,\mathbb{R})$. We use $\hat{\theta}$ to denote the argument of the complex number $\int_{X}(\omega+\sqrt{-1}\alpha)^{n}$, i.e.,
\begin{equation}\label{eq: topAngle}
\int_{X}(\omega+\sqrt{-1}\alpha)^{n} \in \mathbb{R}_{>0}e^{\sqrt{-1}\hat{\theta}}.
\end{equation}
which is well-defined modulo $2\pi$ provided the above integral does not vanish, an assumption we shall make throughout the paper.  The deformed Hermitian-Yang-Mills (dHYM) equation seeks a smooth function $\phi$ on $X$ such that the $(1,1)$ form $\alpha_{\phi} := \alpha+\mn\ddb \phi$ satisfies the non-linear partial differential equation
\begin{equation}\label{dHYM}
\begin{cases}
{\rm Im}\left(e^{-\sqrt{-1}\hat{\theta}}\left(\omega+\sqrt{-1}\alpha_{\phi}\right)^{n}\right) = 0, \\[2mm]
{\rm Re}\left(e^{-\sqrt{-1}\hat{\theta}}\left(\omega+\sqrt{-1}\alpha_{\phi}\right)^{n}\right) > 0.
\end{cases}
\end{equation}
The dHYM equation plays a fundamental role in mirror symmetry \cite{LYZ, MMMS} and its solvability is expected to be related to deep notions of stability in algebraic geometry.  We refer the reader to \cite{CXY} and the references therein for an introduction to the physical and mathematical aspects of the dHYM equation.  Understanding the solvability of the dHYM equation has recently generated a great deal of interest, beginning with the work of Jacob-Yau \cite{JY}, and the second author with Jacob and Yau \cite{CJY15}. Inspired by work of Solomon \cite{Sol}, Thomas \cite{Th} and Thomas-Yau \cite{ThY} in symplectic geometry the second author and Yau \cite{CY18} recently introduced an infinite dimensional GIT (Geometric Invariant Theory) approach to the dHYM equation.  In this approach a fundamental role is played by the following space
\begin{definition}
The space of {\em almost calibrated} $(1,1)$ forms in the class $[\alpha]$ is defined to be
\begin{equation}\label{eq: defnH}
\mathcal{H} = \{\phi\in C^{\infty}(X)~|~
{\rm Re}\left(e^{-\sqrt{-1}\hat{\theta}}(\omega+\sqrt{-1}\alpha_{\phi})^{n}\right)>0\}.
\end{equation}
\end{definition}
The space $\mathcal{H}$ is a (possibly empty) open subset of the space of smooth, real valued functions on $X$, and hence inherits the structure of an infinite dimensional manifold.  Under mirror symmetry the space $\mathcal{H}$ is mirror to the space of positive (or almost calibrated) Lagrangians studied by Solomon \cite{Sol, Solomon14}; this is the motivation for name we have attached to $\mathcal{H}$.  Assuming, as we shall do throughout the paper, that $\mathcal{H}$ is non-empty, we can define a Riemannian structure on $\mathcal{H}$ in the following way; for any $\phi\in\mathcal{H}$, the tangent space $T_{\phi}\mathcal{H} = C^{\infty}(X)$. Define a Riemannian metric on $\mathcal{H}$ by
\[
\langle\psi_{1},\psi_{2}\rangle_{\phi}
= \int_{X}\psi_{1}\psi_{2}{\rm Re}\left(e^{-\sqrt{-1}\hat{\theta}}(\omega+\sqrt{-1}\alpha_{\phi})^{n}\right),
\]
for $\psi_{1},\psi_{2}\in T_{\phi}\mathcal{H}$. Let $\phi(t)$ ($t\in[0,1]$) be a smooth path in $\mathcal{H}$. The length of $\phi$ is given by
\[
{\rm length}(\phi) := \int_{0}^{1}\left(  \int_{X}\dot{\phi}^{2}{\rm Re}\left(e^{-\sqrt{-1}\hat{\theta}}(\omega+\sqrt{-1}\alpha_{\phi})^{n}\right) \right)^{\frac{1}{2}}dt,
\]
where $\dot{\phi}=\frac{\de\phi}{\de t}$. Therefore, for any $\phi_{0},\phi_{1}\in\mathcal{H}$, the Riemannian metric $\mathcal{H}$ defines a ``distance" function on $\mathcal{H}\times\mathcal{H}$:
\[
d(\phi_{0},\phi_{1})
:= \inf\{{\rm length}(\phi)~|~\text{$\phi$ is a smooth path in $\mathcal{H}$ joining $\phi_{0}, \phi_{1}$}\}.
\]
The corresponding geodesic equation is \cite{CY18}
\begin{equation}\label{eq: introGeoEq}
\begin{aligned}
&\ddot{\phi}{\rm Re}\left(e^{-\sqrt{-1}\hat{\theta}}(\omega+\sqrt{-1}\alpha_{\phi})^{n}\right)\\
&+n\sqrt{-1}\de\dot{\phi}\wedge\dbar\dot{\phi}\wedge{\rm Im}\left(e^{-\sqrt{-1}\hat{\theta}}(\omega+\sqrt{-1}\alpha_{\phi})^{n-1}\right)= 0.
\end{aligned}
\end{equation}
This equation is a fully nonlinear degenerate elliptic PDE, and hence in general, we cannot expect the existence of smooth solutions. In other words, for any $\phi_{0},\phi_{1}\in\mathcal{H}$, there may not exist a smooth geodesic between $\phi_0, \phi_1$. Instead, the second author and Yau \cite{CY18} introduced an $\epsilon$-regularized version of geodesic equation which is a fully nonlinear elliptic equation whose solution $\phi^{\epsilon}$ is an approximate geodesic, which we refer to as an $\epsilon$-geodesic.  Assuming that the class $[\alpha]$ satisfies a {\em hypercritical phase} condition (see Section~\ref{Geodesics section} for a definition), the second author and Yau proved the existence of smooth $\epsilon$-geodesics, and weak geodesics with $C^{1,\alpha}$ regularity, for any $\alpha \in (0,1)$.  We remark that a real version of~\eqref{eq: introGeoEq}, originating from Solomon's work in symplectic geometry \cite{Sol} has recently been studied by several groups \cite{DarRu, Del, RuSol}.  In the hypercritical phase case, Jacob extended the techniques of \cite{RuSol} to prove the existence of weak geodesics in the space $\mathcal{H}$ with $C^0$ regularity \cite{Jac}.  The purpose of this paper is to study more detail the space $(\mathcal{H},d)$.   Our first result shows that, in the hypercritical phase case, $d$ is actually a distance function on $\mathcal{H}$ and that the distance $d$ can be approximated by the length of $\epsilon$-geodesics.

\begin{theorem}\label{Metric structure}
Assume that $[\alpha]$ has hypercritical phase.  Then $(\mathcal{H},d)$ is a metric space, and for any $\phi_{0},\phi_{1}\in\mathcal{H}$, we have
\[
d(\phi_{0},\phi_{1}) = \lim_{\epsilon\rightarrow0}{\rm length}(\phi^{\epsilon}),
\]
where $\phi^{\epsilon}$ is the $\epsilon$-geodesic joining $\phi_{0}, \phi_{1}$.  Furthermore, $d: \mathcal{H}\times \mathcal{H} \rightarrow \mathbb{R}$ is $C^1$ differentiable away from the diagonal.
\end{theorem}

In fact, we give an explicit formula for the derivative, which is useful in its own right; see equation~\eqref{eq: distanceDer}.  Next we consider the curvature of the infinite dimensional Riemannian manifold $(\mathcal{H},\langle\cdot,\cdot\rangle)$. We show that there exists a Levi-Civita connection on $(\mathcal{H},\langle\cdot,\cdot\rangle)$ and the corresponding sectional curvature is non-positive.  Besides being of intrinsic interest, this result strengthens the analogy with finite dimensional GIT.

\begin{theorem}\label{Connection and curvature}
The Riemannian manifold $(\mathcal{H},\langle\cdot,\cdot\rangle)$ can be equipped with a Levi-Civita connection, and the sectional curvature is non-positive, i.e., for any $\phi\in\mathcal{H}$ and $\psi,\eta\in T_{\phi}\mathcal{H}=C^{\infty}(X)$, we have
\[
K(\psi,\eta) := \frac{\langle R(\psi,\eta)\eta,\psi\rangle}
{\langle\psi,\psi\rangle \langle\eta,\eta\rangle-\langle\psi,\eta\rangle^{2}} \leq 0.
\]
\end{theorem}

This result is closely related, by mirror symmetry, to a result of Solomon \cite{Solomon14}, which shows that the infinite dimensional manifold of positive Lagrangians admits a metric with negative sectional curvature.  

As a metric space, $(\mathcal{H},d)$ is not complete, and so it is natural to consider its completion $(\widetilde{\mathcal{H}},\ti{d})$. We show that that the non-positive curvature of $(\mathcal{H}, \langle, \cdot, \cdot \rangle)$ carries over to the completion.  Namely, we show that $(\widetilde{\mathcal{H}},\ti{d})$ is a geodesic metric space with non-positive curvature in the sense of Alexandrov. Precisely, we prove 
\begin{theorem}\label{CAT(0) space}
Suppose $[\alpha]$ has hypercritical phase.  Then $(\widetilde{\mathcal{H}},\ti{d})$ is a ${\rm CAT}(0)$ space.
\end{theorem}

One upshot of this result is that, at least when $[\alpha]$ has hypercitical phase, the space $(\mathcal{H}, d)$ has a {\em intrinsically defined} ideal boundary.  According to GIT and the work in \cite{CY18}, the ideal boundary is intimately connected to both the existence of solutions to the dHYM equation, and algebraic stability conditions.



Finally, we obtain an improved regularity result for geodesics in the space $\mathcal{H}$, under the hypercritical phase assumption. 
\begin{theorem}\label{C11 regularity}
Suppose $[\alpha]$ has hypercritical phase.  Then for any $\phi_{0},\phi_{1}\in\mathcal{H}$, the weak geodesic whose existence is established in Theorem~\ref{thm: CY} is $C^{1,1}$.
\end{theorem}

We demonstrate that this result is optimal, by using work of Darvas-Lempert \cite{DL12} to construct examples of points in $\mathcal{H}$ which cannot be joined by a $C^{2}$ geodesic. Recently, the first author, Tosatti and Weinkove \cite{CTW17} proved the $C^{1,1}$ regularity of geodesics in the space of K\"ahler metrics.

It is worth pointing out that there is a largely parallel, but much better developed, theory for the space of K\"ahler metrics equipped with the Donaldson-Mabuchi-Semmes Riemannian metric.  Indeed, if $(X,\omega)$ is a compact K\"ahler manifold and $\mathcal{H}_{PSH} := \{ \phi \in C^{\infty}(X,\mathbb{R}) : \omega + \mn\ddb \phi >0\}$ denotes the space of $\omega$-PSH functions, then the Donaldson-Mabuchi-Semmes Riemannian metric is given by
\[
\langle \psi_1, \psi_2\rangle_{\phi} = \int_{X}\psi_1\psi_2 \omega_{\phi}^n.
\]
With this metric $\mathcal{H}_{PSH}$ is a negatively curved infinite dimensional Riemannian manifold.  The properties of $\mathcal{H}_{PSH}$, as well as its completions (with respect to certain Finsler norms) have played an important role in the study of K\"ahler metrics with constant scalar curvature; we refer the reader to \cite{Chen2000, CalabiChen2002, Blocki2011, Darvas19, GZ, PSS}, and the references therein for an introduction to this circle of ideas.

The organization of this paper is as follows.  Section~\ref{Geodesics section} consists of background concerning the space $\mathcal{H}$ and its geodesics.  Section~\ref{sec: metric} establishes Theorem~\ref{Metric structure}.  In Section~\ref{Connection and curvature section} we introduce the Levi-Civita connection on $\mathcal{H}$, and prove that $\mathcal{H}$ has non-positive sectional curvature, establishing Theorem~\ref{Connection and curvature}.  In Section~\ref{sec: completion} we study the completion of $\mathcal{H}$ with respect to the Riemannian distance, and show that this space is a ${\rm CAT}(0)$ space.  Finally, in Section~\ref{C11 section} we prove that geodesics in the space $\mathcal{H}$ are in fact $C^{1,1}$, and using work of Darvas-Lempert \cite{DL12} we show that this result is sharp by constructing geodesics which are not $C^2$. 

{\bf Acknowledgements}: The authors thank Jake Solomon and Tam\'as Darvas for helpful comments.

\section{Background}\label{Geodesics section}

In this section we recall some of the basic properties of the infinite dimensional Riemannian manifold $\mathcal{H}$.  Given a smooth, real $(1,1)$ form $\alpha$, for any point $p\in X$ we can choose local holomorphic coordinates so that
\[
\omega(p) = \sum_{i} \sqrt{-1}dz_i\wedge d\bar{z}_i \qquad \alpha(p) = \sum_{i} \sqrt{-1}\lambda_i dz_i\wedge d\bar{z}_i 
\]
for $\lambda_i \in \mathbb{R}$, $1\leq i \leq n$.  More intrinsically, $\lambda_i$ are the eigenvalues of the hermitian endomorphism $\omega^{-1}\alpha$. We define the phase operator to be
\begin{equation}\label{eq: phaseOp}
\Theta_{\omega}(\alpha) := \sum_i \arctan(\lambda_i).
\end{equation}
It is straightforward to check that $\Theta_{\omega}(\alpha)$ is a smooth map from $X$ to $(-n\frac{\pi}{2}, n\frac{\pi}{2})$ and that the dHYM equation~\eqref{dHYM} is equivalent to
\[
\begin{aligned}
\Theta_{\omega}(\alpha_{\phi}) = \beta \qquad \text{ where } \beta = \hat{\theta} \mod 2\pi  \text{ is constant }
 \end{aligned}
 \]
and we recall that $\hat{\theta} \in [0, 2\pi)$ is the topological quantity defined by~\eqref{eq: topAngle}.  The space $\mathcal{H}$ defined in~\eqref{eq: defnH} can then be written as
\begin{equation}\label{eq: HdefBranch}
\mathcal{H} = \bigsqcup_{\{\beta \in (-n\frac{\pi}{2}, n \frac{\pi}{2}) : \beta = \hat{\theta} \mod 2\pi \}} \{ \phi \in C^{\infty}(X)~|~ |\Theta_{\omega}(\alpha_{\phi}) - \beta| < \frac{\pi}{2}\}
\end{equation}

An easy argument using the maximum principle shows that either $\mathcal{H}$ is empty, or the disjoint union on the right hand side of~\eqref{eq: HdefBranch} collapses to only one branch \cite{CXY}.  That is, there is a unique $\beta \in (-n\frac{\pi}{2}, n \frac{\pi}{2})$ such that $\beta = \hat{\theta} \mod 2\pi$ and
\[
\{ \phi \in C^{\infty}(X)~|~ |\Theta_{\omega}(\alpha_{\phi}) - \beta| < \frac{\pi}{2}\} \ne \emptyset
\]
In this situation we identify $\hat{\theta}$ with this uniquely defined {\em lifted phase} $\hat{\theta} \in (-n\frac{\pi}{2}, n \frac{\pi}{2})$.  
\begin{definition}
We say that $[\alpha]$ has {\em hypercritical phase} (with respect to $\omega$) if the lifted phase $\hat{\theta} \in ((n-1)\frac{\pi}{2}, n \frac{\pi}{2})$.
\end{definition}

Let $\mathcal{X}=X\times\mathcal{A}$, $\mathcal{A}=\{t\in\mathbb{C}~|~e^{-1}\leq |t| \leq 1\}$ and $\pi$ be the projection from $\mathcal{X}$ to $X$. We use $D, \ov{D}$ to denote the complex differential operators on $\mathcal{X}$ and $\de, \dbar$ to denote the operators on $X$. For a path $\phi$ in $\mathcal{H}$, define a function $\Phi$ on $\mathcal{X}$ by
\[
\Phi(x,t) = \phi(x,-\log|t|).
\]
As noted in \cite{CY18}, the path $\phi$ is a geodesic joining $\phi_{0}$ and $\phi_{1}$ if and only if $\Phi$ solves the following equation
\[
\begin{cases}
{\rm Im}\left(e^{-\sqrt{-1}\hat{\theta}}\left(\pi^{*}\omega
+\sqrt{-1}(\pi^{*}\alpha+\sqrt{-1}D\ov{D}\Phi)\right)^{n+1}\right) = 0, \\[2mm]
{\rm Re}\left(e^{-\sqrt{-1}\hat{\theta}}\left(\omega+\sqrt{-1}(\alpha+\ddbar\Phi)\right)^{n}\right) > 0, \\[2mm]
\Phi(\cdot,1) = \phi_{0}, \ \Phi(\cdot,e^{-1}) = \phi_{1},
\end{cases}
\]
on $\mathcal{X}$.
To study this degenerate elliptic equation, for any $\epsilon>0$, the second author and Yau \cite{CY18} introduced the $\epsilon$-geodesic equation
\begin{equation}\label{epsilon geodesic eqn}
\begin{cases}
{\rm Im}\left(e^{-\sqrt{-1}\hat{\theta}}\left(\pi^{*}\omega+\epsilon^{2}dt\wedge d\ov{t}
+\sqrt{-1}(\pi^{*}\alpha+\sqrt{-1}D\ov{D}\Phi^{\epsilon})\right)^{n+1}\right) = 0, \\[2mm]
{\rm Re}\left(e^{-\sqrt{-1}\hat{\theta}}\left(\omega+\sqrt{-1}(\alpha+\ddbar\Phi)\right)^{n}\right) > 0, \\[2mm]
\Phi^{\epsilon}(\cdot,t)\big|_{\{|t|=1\}} = \phi_{0}, \quad \Phi^{\epsilon}(\cdot,t)\big|_{\{|t|=e^{-1}\}} = \phi_{1}.
\end{cases}
\end{equation}
Introduce the K\"ahler metric $\hat{\omega}_{\epsilon} := \pi^{*}\omega+ \ve^2\sqrt{-1}dt\wedge d\bar{t}$ on $\mathcal{X}$.  Then it is straightforward to check that the $\epsilon$-geodesic equation is equivalent to the PDE
\[
\begin{aligned}
&\Theta_{\hat{\omega}_{\epsilon}} (\pi^{*}\alpha+\sqrt{-1}D\ov{D} \Phi^{\epsilon}) = \hat{\theta}\\
& \Phi^{\epsilon}(\cdot,t)\big|_{\{|t|=1\}} = \phi_{0}, \quad \Phi^{\epsilon}(\cdot,t)\big|_{\{|t|=e^{-1}\}} = \phi_{1},
\end{aligned}
\]
where $\Theta$ is the operator defined in~\eqref{eq: phaseOp}.  An application of the maximum principle shows that the solution $\Phi^{\epsilon}$ is $S^{1}$-invariant, i.e.,
\[
\Phi^{\epsilon}(x,t)  = \Phi^{\epsilon}(x,|t|).
\]
Define a path $\phi^{\epsilon}$ in $\mathcal{H}$ by
\[
\phi^{\epsilon}(x,-\log|t|)=\Phi^{\epsilon}(x,t).
\]
The path $\phi^{\epsilon}$ is said to be the $\epsilon$-geodesic joining $\phi_{0}, \phi_{1}$. In \cite{CY18}, the second author and Yau proved that the $\epsilon$-geodesic equation admits a unique, smooth smooth solution. More precisely, the following result was proved
\begin{theorem}[Collins-Yau, \cite{CY18}]\label{thm: CY}
For any $\phi_0, \phi_1 \in \mathcal{H}$ there exists a unique, $S^1$ invariant solution to~\eqref{epsilon geodesic eqn}, and the following estimate holds:  there is a constant $C$ depending on $\alpha, X, \omega, \phi_0, \phi_1$, but not $\epsilon$, such that
\begin{equation}\label{CY estimate 1}
\sup_{\mathcal{X}}\left(|\Phi^{\epsilon}|+|D\Phi^{\epsilon}|+|D\ov{D}\Phi^{\epsilon}|\right)
+\sup_{\de\mathcal{X}}|D^{2}\Phi^{\epsilon}| \leq C,
\end{equation}
or equivalently
\begin{equation}\label{CY estimate 2}
\sup_{X\times[0,1]}\left(|\phi^{\epsilon}|+|\phi_{t}^{\epsilon}|+|\nabla\phi^{\epsilon}|+|\phi_{tt}^{\epsilon}|+|\nabla\phi_{t}^{\epsilon}|
+|\de\dbar\phi^{\epsilon}|\right)
+\sup_{\de(X\times[0,1])}|\nabla^{2}\phi^{\epsilon}| \leq C.
\end{equation}
As $\epsilon \rightarrow 0$, the paths $\Phi^{\epsilon}$ (or equivalently $\phi^{\epsilon}$) converge to a $C^{1,\alpha}$ geodesic in $\mathcal{H}$.
\end{theorem}
For later use, we derive the $\epsilon$-geodesic equation for $\phi^{\epsilon}$.

\begin{lemma}\label{Eqn for phi epsilon}
The $\epsilon$-geodesic $\phi^{\epsilon}$ satisfies the following equation:
\[
\begin{aligned}
&\ddot{\phi}^{\epsilon} {\rm Re}\left(e^{-\sqrt{-1}\hat{\theta}}(\omega + \sqrt{-1}(\alpha + \sqrt{-1}\ddb \phi^{\epsilon}))^n\right)\\
&+ n\sqrt{-1}\de \dot{\phi}^{\epsilon}\wedge \dbar \dot{\phi}^{\epsilon} \wedge {\rm Im}\left(e^{-\sqrt{-1}\hat{\theta}}(\omega + \sqrt{-1}(\alpha + \sqrt{-1}\ddb \phi^{\epsilon}))^{n-1}\right)\\
& = -4e^{-2t}\epsilon^2{\rm Im}\left(e^{-\sqrt{-1}\hat{\theta}} (\omega+ \sqrt{-1}(\alpha + \sqrt{-1}\ddb \phi^{\epsilon}))^{n}\right).
\end{aligned}
\]
\end{lemma}

\begin{proof}
Recall the $\epsilon$-geodesic equation (\ref{epsilon geodesic eqn}) is given by
\[
{\rm Im}\left(e^{-\sqrt{-1}\hat{\theta}}(\pi^{*}\omega +\epsilon^2\sqrt{-1}dt\wedge d\bar{t} + \sqrt{-1}(\pi^{*}\alpha + \sqrt{-1}D\ov{D} \Phi^{\epsilon}))^{n+1}\right)=0.
\]
Expanding this equation gives
\begin{equation}\label{Eqn for phi epsilon eqn 1}
\begin{aligned}
&{\rm Im}\left(e^{-\sqrt{-1}\hat{\theta}}(\pi^{*}\omega + \sqrt{-1}(\pi^{*}\alpha + \sqrt{-1}D\ov{D} \Phi^{\epsilon}))^{n+1}\right)\\
&+ \epsilon^2{\rm Im}\left(e^{-\sqrt{-1}\hat{\theta}} (n+1)\sqrt{-1}dt\wedge d\bar{t} \wedge (\pi^{*}\omega+ \sqrt{-1}(\pi^{*}\alpha + \sqrt{-1}D\ov{D} \Phi^{\epsilon}))^{n}\right) = 0.
\end{aligned}
\end{equation}
At the same time, by counting the number of $dt, d\bar{t}$ components the term on the second line is equal to
\[
\epsilon^2{\rm Im}\left(e^{-\sqrt{-1}\hat{\theta}} (n+1)\sqrt{-1}dt\wedge d\bar{t} \wedge (\pi^{*}\omega+ \sqrt{-1}(\pi^{*}\alpha +\ddb \Phi^{\epsilon}))^{n}\right).
\]

On the other hand, if we set $s= -\log|t|$ and denote
\[
\dot{\phi^{\epsilon}} = \frac{\de \phi^{\epsilon}}{\de s}, \quad \ddot{\phi^{\epsilon}} = \frac{\de^{2} \phi^{\epsilon}}{\de s^{2}},
\]
then we have
\[
\de_{t}\Phi^{\epsilon} = -\frac{1}{2t}\dot{\phi}^{\epsilon}, \
\de_{\ov{t}}\Phi^{\epsilon} = -\frac{1}{2\ov{t}}\dot{\phi}^{\epsilon}, \
\de_{t}\de_{\ov{t}}\Phi^{\epsilon} = \frac{1}{4|t|^{2}}\dot{\phi}^{\epsilon},  \
\de\dbar\Phi^{\epsilon} = \de\dbar\phi^{\epsilon}.
\]
It then follows that
\[
\begin{aligned}
&{\rm Im}\left(e^{-\sqrt{-1}\hat{\theta}}(\pi^{*}\omega + \sqrt{-1}(\pi^{*}\alpha + \sqrt{-1}D\ov{D} \Phi^{\epsilon}))^{n+1}\right)\\
&=  (n+1)\sqrt{-1}\frac{dt \wedge d\bar{t}}{4|t|^2} \wedge \ddot{\phi}^{\epsilon} {\rm Re}\left(e^{-\sqrt{-1}\hat{\theta}}(\omega + \sqrt{-1}(\alpha + \sqrt{-1}\ddb \phi^{\epsilon}))^n\right)\\
&+ (n+1)\sqrt{-1}\frac{dt \wedge d\bar{t}}{4|t|^2} \wedge n\sqrt{-1}\de \dot{\phi}^{\epsilon}\wedge \dbar \dot{\phi}^{\epsilon} \wedge {\rm Im}\left(e^{-\sqrt{-1}\hat{\theta}}(\omega + \sqrt{-1}(\alpha + \sqrt{-1}\ddb \phi^{\epsilon}))^{n-1}\right).
\end{aligned}
\]
Combining this with (\ref{Eqn for phi epsilon eqn 1}) and $|t|=e^{-s}$, we obtain
\[
\begin{aligned}
&\ddot{\phi}^{\epsilon} {\rm Re}\left(e^{-\sqrt{-1}\hat{\theta}}(\pi^{*}\omega + \sqrt{-1}(\pi^{*}\alpha + \sqrt{-1}\ddb \phi^{\epsilon}))^n\right)\\
&+ n\sqrt{-1}\de \dot{\phi}^{\epsilon}\wedge \dbar \dot{\phi}^{\epsilon} \wedge {\rm Im}\left(e^{-\sqrt{-1}\hat{\theta}}(\pi^{*}\omega + \sqrt{-1}(\pi^{*}\alpha + \sqrt{-1}\ddb \phi^{\epsilon}))^{n-1}\right)\\
& = -4e^{-2s}\epsilon^2{\rm Im}\left(e^{-\sqrt{-1}\hat{\theta}} (\pi^{*}\omega+ \sqrt{-1}(\pi^{*}\alpha + \sqrt{-1}\ddb \phi^{\epsilon}))^{n}\right).
\end{aligned}
\]
Restricting this equation on $X$ and replacing $s$ by $t$, we obtain the lemma.
\end{proof}

Before proceeding we make the following definition, whose only purpose is to ease notation, and shorten some otherwise lengthy formulae.
\begin{definition}\label{defn: notation}
Given $\phi \in C^{\infty}(X)$, we set
\[
\Omega_{\phi} := \omega + \sqrt{-1}\alpha_{\phi}.
\]
\end{definition}

\section{Metric structure of $\mathcal{H}$}\label{sec: metric}

\subsection{Some estimates for $\epsilon$-geodesic}
For any $\phi_0,\phi_1\in\mathcal{H}$, let $\phi^{\epsilon}$ be the $\epsilon$-geodesic joining $\phi_0, \phi_1$. For $t\in[0,1]$, define $E^{\epsilon}(t)$ by
\[
E^{\epsilon}(t) := \int_{X}(\dot{\phi}^{\epsilon})^{2}{\rm Re}\left(e^{-\sqrt{-1}\hat{\theta}}(\omega + \sqrt{-1}(\alpha + \sqrt{-1}\ddb \phi^{\epsilon}))^n\right).
\]
We have the following estimates for $\phi^{\epsilon}$ and $E^{\epsilon}$.
\begin{lemma}\label{Almost time convexity}
There exists a constant $C$, depending only on $\phi_0$, $\phi_1$, $\alpha$ and $(X,\omega)$ such that
\[
\ddot{\phi}^{\epsilon} \geq -C\epsilon^2.
\]
\end{lemma}

\begin{proof}
Recalling the relationship between $\phi^{\epsilon}$ and $\Phi^{\epsilon}$, it suffices to prove
\[
\de_{t}\de_{\ov{t}}\Phi^{\epsilon} \geq -C\epsilon^2.
\]
Let $\mu_{0},\cdots,\mu_{n}$ be the eigenvalues of $\pi^{*}\alpha+\sqrt{-1}D\ov{D}\Phi^{\epsilon}$ with respect to $\pi^*\omega + \epsilon^2\sqrt{-1}dt\wedge d\bar{t}$. Then the $\epsilon$-geodesic equation (\ref{epsilon geodesic eqn}) implies
\[
\sum_{i=0}^{n} \arctan (\mu_i) = \hat{\theta}.
\]
Let $\lambda_{1},\cdots,\lambda_{n}$ be the eigenvalues of $\alpha+\ddbar\Phi^{\epsilon}$ with respect to $\omega$. By the Schur-Horn theorem (see \cite{Horn54}) and \cite[Lemma 3.1 (7)]{CY18}, we have
\[
\arctan\left(\frac{\Phi^{\epsilon}_{t\bar{t}}}{\epsilon^2}\right) +\sum_{i=1}^{n}\arctan(\lambda_{i}) \geq \sum_{i=0}^{n} \arctan (\mu_i) = \hat{\theta}.
\]
Thanks to estimate (\ref{CY estimate 1}),
\[
|\lambda_{i}| \leq C, \quad \text{for $1\leq i \leq n$}.
\]
Thus,
\[
\arctan\left(\frac{\Phi^{\epsilon}_{t\bar{t}}}{\epsilon^2}\right)  > \hat{\theta} - n \arctan (C) > -\frac{\pi}{2},
\]
where we used that $\hat{\theta} \in ((n-1)\frac{\pi}{2}, n \frac{\pi}{2})$.  We can therefore apply tangent to both sides to obtain
\[
\Phi^{\epsilon}_{t\bar{t}} \geq -C\epsilon^2,
\]
as desired.
\end{proof}

\begin{lemma}\label{Estimates for E}
Let $\phi^{\epsilon}(t)$ be an $\epsilon$-geodesic between $\phi_0, \phi_1 \in \mathcal{H}$.  There exists a constant $C$, depending only on $\phi_0$, $\phi_1$, $\alpha$ and $(X,\omega)$ such that
\begin{enumerate}
\item[(i)] $\displaystyle\left|\frac{dE^{\epsilon}}{dt}\right|\leq C\epsilon^{2}$.
\item[(ii)] $E^{\epsilon}(t)$ has the following lower bound:
\[
\begin{split}
E^{\epsilon}(t) \geq {} &
\max \left\{
\int_{\{\phi_{0}>\phi_{1}\}}(\phi_{0}-\phi_{1})^{2}{\rm Re}\left(e^{-\sqrt{-1}\hat{\theta}}(\omega+\sqrt{-1}\alpha_{\phi_{0}})^{n}\right),\right. \\
& \left.\int_{\{\phi_{1}>\phi_{0}\}}(\phi_{1}-\phi_{0})^{2}{\rm Re}\left(e^{-\sqrt{-1}\hat\theta}(\omega + \sqrt{-1}\a_{\phi_1})^{n}\right)\right\}-C\epsilon^{2}.
\end{split}
\]
\end{enumerate}
\end{lemma}

\begin{proof}
We use the notation $\Omega_{\phi}$ introduced in Definition~\ref{defn: notation}. For (i), by direct computation we have
\[
\begin{aligned}
\frac{d E^{\epsilon}}{dt} & = 2\int \dot{\phi}^{\epsilon} \ddot{\phi}^{\epsilon} {\rm Re}\left(e^{-\sqrt{-1}\hat{\theta}}\Omega_{\phi^{\epsilon}}^n\right)+2\int \dot{\phi}^{\epsilon} n\sqrt{-1}\de \dot{\phi}^{\epsilon}\wedge \dbar \dot{\phi}^{\epsilon} \wedge {\rm Im}\left(e^{-\sqrt{-1}\hat{\theta}}\Omega_{\phi^{\epsilon}}^{n-1}\right).
\end{aligned}
\]
Therefore, by Lemma \ref{Eqn for phi epsilon}, we obtain
\[
\frac{d E^{\epsilon}}{dt} = -8e^{-2t}\epsilon^2\int \dot{\phi}^{\epsilon}{\rm Im}\left(e^{-\sqrt{-1}\hat{\theta}} \Omega_{\phi^{\epsilon}}^{n}\right).
\]
Using the estimate (\ref{CY estimate 2}), we have
\[
|\dot{\phi}^{\epsilon}|<C, \quad |\ddb \phi^\epsilon| <C.
\]
It then follows that
\[
\bigg|\int \dot{\phi}^{\epsilon}{\rm Im}\left(e^{-\sqrt{-1}\hat{\theta}}\Omega_{\phi^{\epsilon}}^{n}\right)\bigg| < C,
\]
and hence we obtain (i).

For (ii), by Lemma \ref{Almost time convexity}, we have
\[
\ddot{\phi}^{\epsilon} \geq -C\epsilon^{2},
\]
which implies
\[
\dot{\phi}^{\epsilon}(0) \leq \phi^{\epsilon}(1)-\phi^{\epsilon}(0)+C\epsilon^{2}= \phi_1-\phi_0 + C\epsilon^2.
\]
Together with the bound $|\dot{\phi}^{\epsilon}|<C$, we get that for any point $p\in\{\phi_{0}>\phi_{1}\}$ there holds
\[
\begin{split}
\left(\dot{\phi}^{\epsilon}(p,0)\right)^{2}
\geq {} & \left( \phi_{1}(p)- \phi_{0}(p) \right)^{2}-C\epsilon^{4}.
\end{split}
\]
Thus,
\[
\begin{split}
E^{\epsilon}(0)
= {} &  \int_{X}(\dot{\phi^{\epsilon}}(0))^{2}{\rm Re}\left(e^{-\sqrt{-1}\hat{\theta}}\Omega_{\phi_{0}}^{n}\right) \\
\geq {} & \int_{\{\phi_{0}>\phi_{1}\}}(\phi_{0}-\phi_{1})^{2}{\rm Re}\left(e^{-\sqrt{-1}\hat{\theta}}\Omega_{\phi_{0}}^{n}\right)
-C\epsilon^{4}.
\end{split}
\]
By a similar argument, we obtain
\[
\begin{split}
E^{\epsilon}(1)
\geq {} & \int_{\{\phi_{1}>\phi_{0}\}}(\phi_{1}-\phi_{0})^{2}{\rm Re}\left(e^{-\sqrt{-1}\hat{\theta}}\Omega_{\phi_{1}}^{n}\right)
-C\epsilon^{4}.
\end{split}
\]
Using $\left|\frac{dE^{\epsilon}}{dt}\right|\leq C\epsilon^{2}$, it is clear that
\[
E^{\epsilon}(t) \geq \max(E^{\epsilon}(0) ,E^{\epsilon}(1))-C\epsilon^{2}.
\]
Combining the above estimates, we get (ii).
\end{proof}

\subsection{Proof of Theorem \ref{Metric structure}}
In this subsection, we give the proof of Theorem \ref{Metric structure}. The general structure of the argument follows that of \cite{Blocki2011, Chen2000}. The first step is to prove a weak version of the triangle inequality.
\begin{lemma}\label{triangle}
Suppose $\psi(t)$ is a smooth curve in $\mathcal{H}$, for $t\in [0,1]$.  Fix a point $\tilde{\psi} \notin \psi([0,1])$, $\tilde{\psi}\in \mathcal{H}$.  For each $t \in [0,1]$, lets $\phi(s,t)$ be an $\epsilon$-geodesic joining $\psi(t)$ to $\tilde{\psi}$.  Then there is a constant $C>0$ depending only on $\tilde\psi,\psi(\cdot),(X,\omega)$ and $\a$ so that
\[
{\rm length}(\phi(\cdot, 0)) \leq {\rm length}(\psi) + {\rm length}(\phi(\cdot, 1)) + C\epsilon.
\]
\end{lemma}
\begin{proof}
Define
\[
\ell_1(t) = {\rm length}(\psi\bigg|_{[0,t]}) \qquad \ell_2(t) = {\rm length}(\phi(\cdot, t)).
\]
It suffices to prove $\ell_1' + \ell_2' \geq - C\epsilon$.  We have
\[
\ell_1'(t) = \left[\int_{X} (\dot{\psi}(t))^2 {\rm Re}\left(e^{-\sqrt{-1}\hat{\theta}} \Omega_\psi^{n}\right)\right]^{1/2}.
\]
We also have
\[
\ell_2'(t) = \int_{0}^{1}\frac{1}{2E(s,t)^{1/2}} \de_{t}E(s,t) ds,
\]
where
\[
E(s,t) = \int_{X} (\frac{\de \phi}{\de s})^2 {\rm Re}\left(e^{-\sqrt{-1}\hat{\theta}} \Omega_\phi^{n}\right).
\]
Write $\frac{\de \phi}{\de s} = \phi_s$ and similarly for $t$.  Then we have (using the notation of Definition~\ref{defn: notation})
\[
\begin{aligned}
\frac{\de E(s,t)}{\de t} &= \int_{X} 2\phi_s \phi_{st} {\rm Re}\left(e^{-\sqrt{-1}\hat{\theta}}\Omega_\phi^{n}\right)\\
&\quad +  2\int_{X} \phi_s n\sqrt{-1}\de \phi_{s} \wedge \dbar \phi_t \wedge  {\rm Im}\left(e^{-\sqrt{-1}\hat{\theta}} \Omega_\phi^{n-1}\right).
\end{aligned}
\]
At the same time we have
\[
\begin{aligned}
&\quad \frac{\de}{\de s} \int_{X} \phi_s \phi_t{\rm Re}\left(e^{-\sqrt{-1}\hat{\theta}} \Omega_\phi^{n}\right) \\
&= \int_{X} \phi_s \phi_{st}{\rm Re}\left(e^{-\sqrt{-1}\hat{\theta}} \Omega_\phi^{n}\right)-n  \int_{X} \phi_s \phi_{t}{\rm Im}\left(e^{-\sqrt{-1}\hat{\theta}} \Omega_\phi^{n-1}\right) \wedge \sqrt{-1}\ddb \phi_{s}\\
&\quad +  \int_{X} \phi_{ss} \phi_{t}{\rm Re}\left(e^{-\sqrt{-1}\hat{\theta}} \Omega_\phi^{n}\right).
\end{aligned}
\]
Integrating by parts on the second term and applying the $\epsilon$-geodesic equation we obtain
\[
\begin{aligned}
&\frac{\de}{\de s} \int_{X} \phi_s \phi_t{\rm Re}\left(e^{-\sqrt{-1}\hat{\theta}} \Omega_\phi^{n}\right)  =\frac{1}{2} \frac{\de E(s,t)}{\de t} -e^{-2s}\epsilon^2\int_{X}\phi_t {\rm Im}\left(e^{-\sqrt{-1}\hat{\theta}} \Omega_\phi^{n}\right).
\end{aligned}
\]
Therefore
\begin{align*}
\ell_2' &= \int_{0}^{1}E^{-1/2} \left(\de_{s}\int_{X}\phi_s\phi_t {\rm Re}(e^{-\sqrt{-1}\hat{\theta}}\Omega_{\phi}^n) \right)ds \\
&\quad - \int_{0}^{1}\frac{1}{2E(s,t)^{\frac{1}{2}}} e^{-2s} \epsilon^2 \int_{X}\phi_t {\rm Im}\left(e^{-\sqrt{-1}\hat{\theta}}\Omega_{\phi}^n\right) ds.
\end{align*}
Integration by parts on the the first term yields
\[
\left[E^{-\frac{1}{2}}\int_{X}\phi_s \phi_t{\rm Re}(e^{-\sqrt{-1}\hat{\theta}}\Omega_{\phi}^n)\right]\bigg|_{s=0}^{s=1} - \frac{1}{2}\int_{0}^{1} E^{-3/2} \de_s E \int_{X}\phi_s\phi_t{\rm Re}(e^{-\sqrt{-1}\hat{\theta}}\Omega_{\phi}^n) ds.
\]
Now, by definition we have $\phi_t(1, \cdot) =0$, $\phi_{t}(0, \cdot) = \de_t \psi$.  If we set $\eta = \phi_s(0,\cdot)$, then we have
\begin{equation}\label{eq: diffDist}
\begin{aligned}
\ell_2' &=- \frac{\int_{X}\eta \psi_t {\rm Re}(e^{-\sqrt{-1}\hat{\theta}}\Omega_{\psi}^n)}{\left(\int_{X} \eta^2  {\rm Re}(e^{-\sqrt{-1}\hat{\theta}}\Omega_{\psi}^n)\right)^{\frac{1}{2}}}- \frac{1}{2}\int_{0}^{1} E^{-3/2} \de_s E \int_{X}\phi_s\phi_t{\rm Re}(e^{-\sqrt{-1}\hat{\theta}}\Omega_{\phi}^n) ds\\
& -\int_{0}^{1}\frac{1}{2E(s,t)^{\frac{1}{2}}} e^{-2s} \epsilon^2 \int_{X}\phi_t {\rm Im}\left(e^{-\sqrt{-1}\hat{\theta}}\Omega_\phi^n\right) ds.
\end{aligned}
\end{equation}
Using the Cauchy-Schwarz inequality we have
\[
\begin{aligned}
\ell_2' &\geq - \left(\int_{X} \psi_t^2 {\rm Re}(e^{-\sqrt{-1}\hat{\theta}}\Omega_{\psi}^n)\right)^{\frac{1}{2}}- \frac{1}{2}\int_{0}^{1} E^{-3/2} \de_s E \int_{X}\phi_s\phi_t{\rm Re}(e^{-\sqrt{-1}\hat{\theta}}\Omega_{\phi}^n) ds\\
& -\int_{0}^{1}\frac{1}{2E(s,t)^{\frac{1}{2}}} e^{-2s} \epsilon^2 \int_{X}\phi_t {\rm Im}\left(e^{-\sqrt{-1}\hat{\theta}}\Omega_\phi^n\right) ds.
\end{aligned}
\]
Therefore, by the formula for $E_s$ we have
\[
\begin{aligned}
\ell_1'+\ell_2' &\geq \int_{0}^{1} E^{-3/2}e^{-2s} \epsilon^2 \int_{X} \phi_s {\rm Re}(e^{-\sqrt{-1}\hat{\theta}}\Omega_{\phi}^n)\int_{X}\phi_s\phi_t{\rm Re}(e^{-\sqrt{-1}\hat{\theta}}\Omega_{\phi}^n) ds \\
& -\int_{0}^{1}\frac{1}{2E(s,t)^{\frac{1}{2}}} e^{-2s} \epsilon^2 \int_{X}\phi_t {\rm Im}\left(e^{-\sqrt{-1}\hat{\theta}}\Omega_\phi^n\right) ds.
\end{aligned}
\]

On the other hand, by Lemma \ref{Estimates for E} (ii), we have $E>c>0$ for $\epsilon$ sufficiently small.  Furthermore, $\phi_s$ uniformly bounded by the uniform estimates for $\epsilon$-geodesics in Theorem~\ref{thm: CY}.  We claim that $\phi_t$ is uniformly bounded by the maximum principle.  To see this observe that the associated $S^{1}$-invariant functions $\Phi(s,t)$ on $\mathcal{X}$ yield a $t$-dependent family of solutions to the $\epsilon$-geodesic equation, which is elliptic.  Differentiating in $T$ shows that $\de_t\Phi$ solves the linearized $\epsilon$-geodesic equation with boundary data $0$ and $\de_t \psi$. The result now follows from the maximum principle.     The uniform estimates in Theorem~\ref{thm: CY} also imply an upper bound for ${\rm Re}(e^{-\sqrt{-1}\hat{\theta}}\Omega_{\phi})$, and so the result follows.
\end{proof}

\begin{lemma}\label{classical-metric-relation}
There is a constant $C>0$ depending only on $[\alpha], [\omega]$ such that, for any $\phi_0,\phi_1\in \mathcal{H}$ we have,
\[
\limsup_{\ve\rightarrow 0^+}{\rm length}(\phi^\ve)\leq C\|\phi_0-\phi_1\|_\infty
\]
where $\phi^\ve(s),s\in [0,1]$ is the $\ve$-geodesic from $\phi_0$ to $\phi_1$.
\end{lemma}
\begin{proof}
To begin, since $\phi_0, \phi_1 \in \mathcal{H}$ we may fix a constant $\ve_0>0$ such that
\begin{equation}
\Re\left(e^{-\sqrt{-1}\hat\theta} (\omega+\sqrt{-1}\alpha_{\phi_i})^n \right)>\ve_0
\end{equation}
for $i=0,1$. We will estimate $\dot{\phi}^{\ve}$ uniformly.  By Lemma \ref{Almost time convexity},
\[
\dot\phi^\ve(0) \leq \dot\phi^\ve(s )+sC\ve^2  \leq \dot\phi^\ve(1)+C\ve^2
\]
on $[0,1]$. Therefore, it suffices to estimate the lower bound of $\dot\phi^\ve(0)$ and the upper bound of $\dot\phi^\ve(1)$. We will work with $\Phi^\ve(t),t\in [e^{-1},1]$ instead. Define
\[
v(t)=A\ve^2 (|t|^2-e^{-2})+2B\log (e|t|)
\]
where $A,B$ to be determined. Furthermore, we have $v_{t\bar t}=A\ve^2$. Hence for $\hat\omega_{\epsilon}=\pi^*\omega+\ve^2 \sqrt{-1}dt\wedge d\bar t$,
\begin{equation}
\begin{split}
\Theta_{\hat\omega_{\ve}}( \alpha + \sqrt{-1}D\ov{D}(\phi_1+v))&=\Theta_\omega(\alpha_{\phi_1})+ \arctan(A)\\
&\geq \hat\theta+\delta_{\ve_0}-\frac{\pi}{2}+\arctan(A).
\end{split}
\end{equation}
Therefore we can choose $A$ sufficiently large depending only on $\ve_0$ such that
\[
\Theta_{\hat\omega_{\ve}}( \alpha + \sqrt{-1}D\ov{D}(\phi_1+v))> \hat\theta=\Theta_{\hat\omega_{\ve}}( \alpha + \sqrt{-1}D\ov{D}\Phi^{\ve}).
\]
We then choose $B$ depending on $A, \ve$ so that  $A\ve^2 (1-e^{-2})+2B=-\|\phi_1-\phi_0\|_\infty$.  Note that as $\epsilon \rightarrow 0$ we have $B\rightarrow -\frac{1}{2}\|\phi_1 -\phi_0\|_{\infty}$.  With these choices $\Phi^\ve \geq \phi_1+v$ on $\partial \mathcal{X}$, and $\Phi^{\ve} = \phi_1 +v$ when $|t|=\epsilon^{-1}$.  By the comparison principle we conclude that $\phi_1+v \leq \Phi^\ve$ on $\mathcal{X}$ and hence for $s=-\log |t|$, $\phi_1+A\ve^2 (e^{-2s}-e^{-2})-2Bs+2B \leq \phi^\ve(s)$ on $X\times [0,1]$. Therefore,
\begin{equation}
\begin{split}
\dot\phi^\ve (1)&\leq \frac{d}{ds}\Big|_{s=1} \left(A\ve^2\cdot e^{-2s}-B s  \right)\\
&=-2e A\ve^2-B.
\end{split}
\end{equation}
The lower bound of $\dot\phi^\ve (0)$ is similar.  Plugging these estimates into the definition of $\rm{length}(\phi^\ve)$ yields the result.
\end{proof}

\begin{proposition}\label{distance-geodesic}
For any $\phi_{0},\phi_{1}\in\mathcal{H}$, we have
\[
d(\phi_{0},\phi_{1}) = \lim_{\epsilon\rightarrow0}{\rm length}(\phi^{\epsilon}),
\]
where $\phi^{\epsilon}$ is the $\epsilon$-geodesic joining $\phi_{0}, \phi_{1}$.
\end{proposition}

\begin{proof}
This follows from the argument in the proof of \cite[Theorem 15]{Blocki2011} using Lemma \ref{classical-metric-relation} and Lemma \ref{triangle}.
\end{proof}

We state an immediate Corollary of Proposition~\ref{distance-geodesic} which will be helpful later.  In essence, this corollary says that (weak) geodesics have constant speed.
\begin{corollary}\label{cor: constSpeed}
Suppose $\phi_0, \phi_1 \in \mathcal{H}$, and let $\phi^{\epsilon}(t)$, $t\in [0,1]$ be an $\epsilon$-geodesic from $\phi_0$ to $\phi_1$.  Then we have
\[
\lim_{\epsilon \rightarrow 0}  \|\dot{\phi}^{\epsilon}(t)\|^2 = d(\phi_0, \phi_1)^2.
\]
\end{corollary}

\begin{proof}
The corollary is an easy consequence of Lemma~\ref{Estimates for E} and Proposition~\ref{distance-geodesic}.  Let
\[
E^{\epsilon}(t) =\|\dot{\phi}^{\epsilon}(t)\|^2 = \int_{X} (\dot{\phi}^{\epsilon}(t))^2 \Ren.
\]
By the uniform estimates of Theorem~\ref{thm: CY} and Lemma~\ref{Estimates for E} (i) there is a constant $C$ independent of $\epsilon$ so that $|E^{\epsilon}(t)| \leq C$ and $|\frac{d E^{\epsilon}}{dt}| \leq C\epsilon^2$.  By Arzela-Ascoli, after passing to a subsequence we have $E^{\epsilon}(t) \rightarrow A$ for some constant $A$.  On the other hand, by Proposition~\ref{distance-geodesic} we have
\[
d(\phi_0, \phi_1) = \lim_{\epsilon \rightarrow 0} \int_0^1 \sqrt{E^{\epsilon}(t)} dt = \sqrt{A}
\]
thus $E^{\epsilon}(t) \rightarrow d(\phi_0,\phi_1)^2$ as $\epsilon \rightarrow 0$ and the result follows.
\end{proof}

The next proposition gives a lower bound for $d(\phi_0, \phi_1)$, establishing that $d$ is a non-degenerate distance function.
\begin{proposition}\label{positivity}
For any $\phi_{0},\phi_{1}\in\mathcal{H}$, we have
\[
\begin{split}
d(\phi_{0},\phi_{1}) ^{2} \geq {} &
\max \left(
\int_{\{\phi_{0}>\phi_{1}\}}(\phi_{0}-\phi_{1})^{2}{\rm Re}\left(e^{-\sqrt{-1}\hat{\theta}}(\omega+\sqrt{-1}\alpha_{\phi_{0}})^{n}\right),\right. \\
& \left.\int_{\{\phi_{1}>\phi_{0}\}}(\phi_{1}-\phi_{0})^{2}{\rm Re}\left(e^{-\sqrt{-1}\hat{\theta}}(\omega+\sqrt{-1}\alpha_{\phi_{1}})^{n}\right)\right).
\end{split}
\]
In particular if $\phi_0 \ne \phi_1\in \mathcal{H}$ are distinct, then $d(\phi_0,\phi_1) >0$.
\end{proposition}

\begin{proof}
Let $\phi^{\epsilon}$ be the $\epsilon$-geodesic joining $\phi_{0}, \phi_{1}$. By Lemma \ref{Estimates for E}, we have
\[
\begin{split}
\left({\rm length}(\phi^{\epsilon})\right)^{2}
= {} & \left(\int_{0}^{1}\sqrt{E^{\epsilon}(t)}dt\right)^{2} \\
\geq {} &\max \left(
\int_{\{\phi_{0}>\phi_{1}\}}(\phi_{0}-\phi_{1})^{2}
{\rm Re}\left(e^{-\sqrt{-1}\hat{\theta}}(\omega+\sqrt{-1}\alpha_{\phi_{0}})^{n}\right),\right. \\
& \left.\int_{\{\phi_{1}>\phi_{0}\}}(\phi_{1}-\phi_{0})^{2}
{\rm Re}\left(e^{-\sqrt{-1}\hat{\theta}}(\omega+\sqrt{-1}\alpha_{\phi_{1}})^{n}\right)\right)-C\epsilon^{2}.
\end{split}
\]
Combining this with Proposition \ref{distance-geodesic} and letting $\epsilon\rightarrow0$, we obtain the proposition.
\end{proof}


Now we are in a position to prove Theorem \ref{Metric structure}.

\begin{proof}[Proof of Theorem \ref{Metric structure}]
By Proposition \ref{distance-geodesic}, it suffices to prove $(\mathcal{H},d)$ is a metric space. The positivity of $d$ is a consequence of Proposition \ref{positivity}. The triangle inequality follows from Lemma \ref{triangle} and Proposition \ref{distance-geodesic}.  It remains only to prove the differentiability.  Fix a point $\phi_0 \in \mathcal{H}$ and let $\psi(t)$ be a smooth curve in $\mathcal{H}$ with $\psi(0) \ne \phi_0$.  For each $t$ let $\phi^{\epsilon}(s,t)$, $s\in[0,1]$ be an $\epsilon$-geodesic from $\psi(t)$ to $\phi_0$.  In the proof of Lemma~\ref{triangle} we proved the following result for $\ell_{\epsilon}(t) = {\rm length}(\phi^{\epsilon}(s,t))$ (see ~\eqref{eq: diffDist})
\begin{equation}\label{eq: distDiff2}
\begin{aligned}
\ell_{\epsilon}'(t) &=- \frac{\int_{X}\phi^{\epsilon}_s(0,t) \psi_t(t) {\rm Re}(e^{-\sqrt{-1}\hat{\theta}}\Omega_{\psi}^n)}{\left(\int_{X} \phi^{\epsilon}_s(0,t)^2  {\rm Re}(e^{-\sqrt{-1}\hat{\theta}}\Omega_{\psi}^n)\right)^{\frac{1}{2}}}\\
&- \frac{1}{2}\int_{0}^{1} (E^{\epsilon})^{-3/2} \de_s E^{\epsilon} \int_{X}\phi^{\epsilon}_s\phi^{\epsilon}_t{\rm Re}(e^{-\sqrt{-1}\hat{\theta}}\Omega_{\phi^{\epsilon}}^n) ds\\
& -\int_{0}^{1}\frac{1}{2E^{\epsilon}(s,t)^{\frac{1}{2}}} e^{-2s} \epsilon^2 \int_{X}\phi^{\epsilon}_t {\rm Im}\left(e^{-\sqrt{-1}\hat{\theta}}\Omega_{\phi^{\epsilon}}^n\right) ds.
\end{aligned}
\end{equation}
As in the proof of Lemma~\ref{triangle} the maximum principle gives the estimate $|\phi^{\epsilon}_{t}(s,t)| \leq C$ for a uniform constant $C$ independent of $\epsilon$.  Moreover, by Lemma~\ref{Estimates for E} (i) we have
\[
|\de_sE| \leq C \epsilon^2
\]
for a uniform constant $C$.  By Corollary~\ref{cor: constSpeed} $E^{\epsilon}(s,t) \rightarrow d(\phi_0, \psi(t)) >0$ (for $|t|$ sufficiently small) as $\epsilon \rightarrow 0$.  Thus, by the uniform estimates for $\epsilon$-geodesics from Theorem~\ref{thm: CY} we have
\[
\begin{aligned}
\bigg|  \frac{1}{2}\int_{0}^{1} (E^{\epsilon})^{-3/2} \de_s E^{\epsilon} \int_{X}\phi^{\epsilon}_s\phi^{\epsilon}_t{\rm Re}(e^{-\sqrt{-1}\hat{\theta}}\Omega_{\phi^{\epsilon}}^n) ds\bigg| &\leq C\epsilon^2\\
\bigg|\int_{0}^{1}\frac{1}{2E^{\epsilon}(s,t)^{\frac{1}{2}}} e^{-2s} \epsilon^2 \int_{X}\phi^{\epsilon}_t {\rm Im}\left(e^{-\sqrt{-1}\hat{\theta}}\Omega_{\phi^{\epsilon}}^n\right) ds\bigg| &\leq C\epsilon^2
\end{aligned}
\]
Integrating~\eqref{eq: distDiff2} from $0$ to $0<t_0 \ll1$ and using the above estimates yields
\[
\frac{1}{t_0} \ell_{\epsilon}(t_0)=- \frac{1}{t_0}\int_{0}^{t_{0}}\frac{\int_{X}\phi^{\epsilon}_s(0,t) \psi_t(t) {\rm Re}(e^{-\sqrt{-1}\hat{\theta}}\Omega_{\psi}^n)}{\left(\int_{X} \phi^{\epsilon}_s(0,t)^2  {\rm Re}(e^{-\sqrt{-1}\hat{\theta}}\Omega_{\psi}^n)\right)^{\frac{1}{2}}}dt - C\epsilon^2
\]
Taking the limit as $\epsilon\rightarrow 0$, using Proposition~\ref{distance-geodesic} and Theorem~\ref{thm: CY} yields
\[
\frac{1}{t_0} d(\psi(t_0), \phi_0) = - \frac{1}{t_0}\int_{0}^{t_0}\frac{\int_{X}\phi_s(0,t) \psi_t(t) {\rm Re}(e^{-\sqrt{-1}\hat{\theta}}\Omega_{\psi}^n)}{\left(\int_{X} \phi_s(0,t)^2  {\rm Re}(e^{-\sqrt{-1}\hat{\theta}}\Omega_{\psi}^n)\right)^{\frac{1}{2}}} dt
\]
where $\phi(s,t)$ is the $C^{1,\alpha}$ geodesic from $\psi(t)$ to $\phi_0$.  Taking the limit as $t_0\rightarrow 0$ yields the result
\begin{equation}\label{eq: distanceDer}
\frac{d}{dt}\big|_{t=0} d(\phi_0, \psi(t)) = \frac{\int_{X}\phi_s(0) \psi_t(t) {\rm Re}(e^{-\sqrt{-1}\hat{\theta}}\Omega_{\psi}^n)}{\left(\int_{X} \phi_s(0)^2  {\rm Re}(e^{-\sqrt{-1}\hat{\theta}}\Omega_{\psi}^n)\right)^{\frac{1}{2}}} 
\end{equation}
where $\phi(s)$ is the geodesic from $\psi(0)$ to $\phi_0$.
\end{proof}

\section{Levi-Civita connection and curvature}\label{Connection and curvature section}

\subsection{Levi-Civita connection}
First, let us recall the Riemannian structure, for any $\phi\in\mathcal{H}$ and $\psi_{1},\psi_{2}\in C^{\infty}(X)$, we have
\[
\langle\psi_{1},\psi_{2}\rangle_{\phi}
= \int_{X}\psi_{1}\psi_{2}{\rm Re}\left(e^{-\sqrt{-1}\hat{\theta}}(\omega+\sqrt{-1}\alpha_{\phi})^{n}\right).
\]
We hope to define a connection $\nabla$ which is compatible with the above Riemannian structure, i.e., for a smooth path $\phi(t)\in\mathcal{H}$, we have
\[
\frac{\de}{\de t}\langle\psi_{1},\psi_{2}\rangle_{\phi}
= \langle\nabla_{\dot{\phi}}\psi_{1},\psi_{2}\rangle_{\phi}
+\langle\psi_{1},\nabla_{\dot{\phi}}\psi_{2}\rangle_{\phi}.
\]
We compute (using again the notation of Definition~\ref{defn: notation})
\begin{equation}\label{Connection eqn 1}
\begin{split}
\frac{\de}{\de t}\langle\psi_{1},\psi_{2}\rangle_{\phi}
= {} & \int_{X}(\dot{\psi}_{1}\psi_{2}+\psi_{1}\dot{\psi}_{2})
{\rm Re}\left(e^{-\sqrt{-1}\hat{\theta}}\Omega_{\phi}^{n}\right) \\
& +\int_{X}\psi_{1}\psi_{2}{\rm Re}\left(n\sqrt{-1}e^{-\sqrt{-1}\hat{\theta}}\Omega_{\phi}^{n-1}\wedge\ddbar\dot{\phi}\right).
\end{split}
\end{equation}
To deal with the second term on the right hand side, we observe
\[
\begin{split}
& \int_{X}(\psi_{1}\psi_{2})\wedge\ddbar\dot{\phi}\wedge\Omega_{\phi}^{n-1}\\
&= -\frac{1}{2}\int_{X}\left(\sqrt{-1}\de(\psi_{1}\psi_{2})\wedge\dbar\dot{\phi}+\sqrt{-1}\de\dot{\phi}\wedge\dbar(\psi_{1}\psi_{2})\right)
\wedge\Omega_{\phi}^{n-1},
\end{split}
\]
and the term
\[
\left(\sqrt{-1}\de(\psi_{1}\psi_{2})\wedge\dbar\dot{\phi}+\sqrt{-1}\de\dot{\phi}\wedge\dbar(\psi_{1}\psi_{2})\right)
\]
is real. Thus,
\[
\begin{split}
& \int_{X}\psi_{1}\psi_{2}{\rm Re}\left(n\sqrt{-1}e^{-\sqrt{-1}\hat{\theta}}\Omega_{\phi}^{n-1}\wedge\ddbar\dot{\phi}\right) \\
= {} & {\rm Re}\left(n\sqrt{-1}e^{-\sqrt{-1}\hat{\theta}} \int_{X}(\psi_{1}\psi_{2})\wedge\ddbar\dot{\phi}\wedge\Omega_{\phi}^{n-1}\right) \\
= {} & -{\rm Im}\left(ne^{-\sqrt{-1}\hat{\theta}} \int_{X}(\psi_{1}\psi_{2})\wedge\ddbar\dot{\phi}\wedge\Omega_\phi^{n-1}\right) \\
= {} & \frac{n}{2}{\rm Im}\left(e^{-\sqrt{-1}\hat{\theta}}\int_{X}\left[\sqrt{-1}\de(\psi_{1}\psi_{2})\wedge\dbar\dot{\phi}+\sqrt{-1}\de\dot{\phi}\wedge\dbar(\psi_{1}\psi_{2})\right]
\wedge\Omega_{\phi}^{n-1}\right) \\
= {} & \frac{n}{2}\int_{X}\left[\sqrt{-1}\de(\psi_{1}\psi_{2})\wedge\dbar\dot{\phi}+\sqrt{-1}\de\dot{\phi}\wedge\dbar(\psi_{1}\psi_{2})\right]
\wedge{\rm Im}\left(e^{-\sqrt{-1}\hat{\theta}}\Omega_{\phi}^{n-1}\right) \\
= {} & \frac{n}{2}\int_{X}\psi_{2}\left(\sqrt{-1}\de\psi_{1}\wedge\dbar\dot{\phi}+\sqrt{-1}\de\dot{\phi}\wedge\dbar\psi_{1}\right)
\wedge{\rm Im}\left(e^{-\sqrt{-1}\hat{\theta}}\Omega_{\phi}^{n-1}\right) \\
& +\frac{n}{2}\int_{X}\psi_{1}\left(\sqrt{-1}\de\psi_{2}\wedge\dbar\dot{\phi}+\sqrt{-1}\de\dot{\phi}\wedge\dbar\psi_{2}\right)
\wedge{\rm Im}\left(e^{-\sqrt{-1}\hat{\theta}}\Omega_{\phi}^{n-1}\right).
\end{split}
\]
Substituting this into (\ref{Connection eqn 1}), we see that
\[
\frac{\de}{\de t}\langle\psi_{1},\psi_{2}\rangle_{\phi}
= \langle \dot{\psi}_{1}+Q(\nabla\dot{\phi},\nabla\psi_{1}), \psi_{2} \rangle_{\phi}
+ \langle \psi_{1}, \dot{\psi}_{2}+Q(\nabla\dot{\phi},\nabla\psi_{2}) \rangle_{\phi},
\]
where
\[
Q(\nabla\psi,\nabla\eta) :=
\frac{n}{2}
\frac{
(\sqrt{-1}\de\psi\wedge\dbar\eta+\sqrt{-1}\de\eta\wedge\dbar\psi)
\wedge{\rm Im}\left(e^{-\sqrt{-1}\hat{\theta}}(\omega+\sqrt{-1}\alpha_{\phi})^{n-1}\right)}
{{\rm Re}\left(e^{-\sqrt{-1}\hat{\theta}}(\omega+\sqrt{-1}\alpha_{\phi})^{n}\right)},
\]
for $\psi,\eta\in C^{\infty}(X)$. Then we have the following definition.
\begin{definition}\label{Levi-Civita connection}
For a smooth path $\phi\in\mathcal{H}$ and $\psi\in C^{\infty}(X)$, the Levi-Civita connection is defined by
\[
\nabla_{\dot{\phi}}\psi
= \dot{\psi}+Q(\nabla\dot{\phi},\nabla\psi).
\]
\end{definition}

\subsection{Curvature Operator}
Let $\phi$ be a smooth map from $[0,1]\times[0,1]$ to $\mathcal{H}$ and $\eta\in C^{\infty}(X)$. We write $\phi_{t}=\frac{\de\phi}{\de t}$, $\phi_{s}=\frac{\de\phi}{\de s}$ and $\phi_{st}=\phi_{ts}=\frac{\de^{2}\phi}{\de s\de t}$. In this subsection, we aim to obtain an explicit expression of $R(\phi_{t},\phi_{s})\eta$. 
By the definition of Levi-Civita connection $\nabla$ (see Definition \ref{Levi-Civita connection}), we have
\[
\begin{split}
R(\phi_{t},\phi_{s})\eta = {} & \nabla_{\phi_{t}}\nabla_{\phi_{s}}\eta-\nabla_{\phi_{s}}\nabla_{\phi_{t}}\eta \\
= {} & \nabla_{\phi_{t}}\left(\eta_{s}+Q(\nabla\phi_{s},\nabla\eta)\right)
-\nabla_{\phi_{s}}\left(\eta_{t}+Q(\nabla\phi_{t},\nabla\eta)\right) \\
= {} & \eta_{st}+Q(\nabla\phi_{t},\nabla\eta_{s})+\frac{\de}{\de t}Q(\nabla\phi_{s},\nabla\eta)
+Q(\nabla\phi_{t},\nabla Q(\nabla\phi_{s},\nabla\eta)) \\
& -\eta_{ts}-Q(\nabla\phi_{s},\nabla\eta_{t})-\frac{\de}{\de s}Q(\nabla\phi_{t},\nabla\eta)
-Q(\nabla\phi_{s},\nabla Q(\nabla\phi_{t},\nabla\eta)).
\end{split}
\]
We compute
\[
\begin{split}
\frac{\de}{\de t}Q(\nabla\phi_{s},\nabla\eta)
= {} & \frac{\de}{\de t}\left[\frac{n}{2}
\frac{
(\sqrt{-1}\de\phi_{s}\wedge\dbar\eta+\sqrt{-1}\de\eta\wedge\dbar\phi_{s})
\wedge{\rm Im}\left(e^{-\sqrt{-1}\hat{\theta}}\Omega_{\phi}^{n-1}\right)}
{{\rm Re}\left(e^{-\sqrt{-1}\hat{\theta}}\Omega_{\phi}^{n}\right)}\right] \\
= {} & Q(\nabla\phi_{st},\nabla\eta)+Q(\nabla\phi_{s},\nabla\eta_{t}) \\
& +\frac{n}{2}(\sqrt{-1}\de\phi_{s}\wedge\dbar\eta+\sqrt{-1}\de\eta\wedge\dbar\phi_{s})
\wedge\frac{\de}{\de t}\frac{{\rm Im}\left(e^{-\sqrt{-1}\hat{\theta}}\Omega_{\phi}^{n-1}\right)}
{{\rm Re}\left(e^{-\sqrt{-1}\hat{\theta}}\Omega_{\phi}^{n}\right)}.
\end{split}
\]
Since
\[
\begin{split}
 \frac{\de}{\de t}{\rm Im}\left(e^{-\sqrt{-1}\hat{\theta}}\Omega_{\phi}^{n-1}\right) ={} & {\rm Im}\left((n-1)\sqrt{-1}e^{-\sqrt{-1}\hat{\theta}}\Omega_{\phi}^{n-2}\wedge\ddbar\phi_{t}\right) \\
= {} & (n-1)\ddbar\phi_{t}\wedge{\rm Re}\left(e^{-\sqrt{-1}\hat{\theta}}\Omega_{\phi}^{n-2}\right)
\end{split}
\]
and
\[
\begin{split}
 \frac{\de}{\de t}{\rm Re}\left(e^{-\sqrt{-1}\hat{\theta}}\Omega_{\phi}^{n}\right) = {} & {\rm Re}\left(n\sqrt{-1}e^{-\sqrt{-1}\hat{\theta}}\Omega_{\phi}^{n-1}\wedge\ddbar\phi_{t}\right) \\
= {} & -n\ddbar\phi_{t}\wedge
{\rm Im}\left(e^{-\sqrt{-1}\hat{\theta}}\Omega_{\phi}^{n-1}\right),
\end{split}
\]
we have
\[
\begin{split}
& \frac{\de}{\de t}Q(\nabla\phi_{s},\nabla\eta) \\[2mm]
= {} & Q(\nabla\phi_{st},\nabla\eta)+Q(\nabla\phi_{s},\nabla\eta_{t}) \\
& +\frac{n}{2}(\sqrt{-1}\de\phi_{s}\wedge\dbar\eta+\sqrt{-1}\de\eta\wedge\dbar\phi_{s})
\wedge\frac{(n-1)\ddbar\phi_{t}\wedge{\rm Re}\left(e^{-\sqrt{-1}\hat{\theta}}\Omega_{\phi}^{n-2}\right)}
{{\rm Re}\left(e^{-\sqrt{-1}\hat{\theta}}\Omega_{\phi}^{n}\right)} \\
& -\frac{n}{2}(\sqrt{-1}\de\phi_{s}\wedge\dbar\eta+\sqrt{-1}\de\eta\wedge\dbar\phi_{s})
\wedge\frac{{\rm Im}\left(e^{-\sqrt{-1}\hat{\theta}}\Omega_{\phi}^{n-1}\right)}
{{\rm Re}\left(e^{-\sqrt{-1}\hat{\theta}}\Omega_{\phi}^{n}\right)}
\cdot\frac{\de_{t}{\rm Re}\left(e^{-\sqrt{-1}\hat{\theta}}\Omega_{\phi}^{n}\right)}{{\rm Re}\left(e^{-\sqrt{-1}\hat{\theta}}\Omega_{\phi}^{n}\right)} \\
= {} & Q(\nabla\phi_{st},\nabla\eta)+Q(\nabla\phi_{s},\nabla\eta_{t}) \\
& +\frac{n}{2}(\sqrt{-1}\de\phi_{s}\wedge\dbar\eta+\sqrt{-1}\de\eta\wedge\dbar\phi_{s})
\wedge\frac{(n-1)\ddbar\phi_{t}\wedge{\rm Re}\left(e^{-\sqrt{-1}\hat{\theta}}\Omega_{\phi}^{n-2}\right)}
{{\rm Re}\left(e^{-\sqrt{-1}\hat{\theta}}\Omega_{\phi}^{n}\right)} \\
& +Q(\nabla\phi_{s},\nabla\eta)
\frac{n\ddbar\phi_{t}\wedge{\rm Im}\left(e^{-\sqrt{-1}\hat{\theta}}\Omega_{\phi}^{n-1}\right)}
{{\rm Re}\left(e^{-\sqrt{-1}\hat{\theta}}\Omega_{\phi}^{n}\right)}.
\end{split}
\]
For $\frac{\de}{\de s}Q(\nabla\phi_{t},\nabla\eta)$, switching $t$ and $s$, we have similar expression:
\[
\begin{split}
& \frac{\de}{\de s}Q(\nabla\phi_{t},\nabla\eta) \\[2mm]
= {} & Q(\nabla\phi_{ts},\nabla\eta)+Q(\nabla\phi_{t},\nabla\eta_{s}) \\
& +\frac{n}{2}(\sqrt{-1}\de\phi_{t}\wedge\dbar\eta+\sqrt{-1}\de\eta\wedge\dbar\phi_{t})
\wedge\frac{(n-1)\ddbar\phi_{s}\wedge{\rm Re}\left(e^{-\sqrt{-1}\hat{\theta}}\Omega_{\phi}^{n-2}\right)}
{{\rm Re}\left(e^{-\sqrt{-1}\hat{\theta}}\Omega_{\phi}^{n}\right)} \\
& +Q(\nabla\phi_{t},\nabla\eta)
\frac{n\ddbar\phi_{s}\wedge{\rm Im}\left(e^{-\sqrt{-1}\hat{\theta}}\Omega_{\phi}^{n-1}\right)}
{{\rm Re}\left(e^{-\sqrt{-1}\hat{\theta}}\Omega_{\phi}^{n}\right)}.
\end{split}
\]
Combining the above equations, we obtain the expression of $R(\phi_{t},\phi_{s})\eta$:
\begin{equation}\label{Expression of R 1}
\begin{split}
& R(\phi_{t},\phi_{s})\eta \\
= {} & \frac{n(n-1)}{2}(\sqrt{-1}\de\phi_{s}\wedge\dbar\eta+\sqrt{-1}\de\eta\wedge\dbar\phi_{s})
\wedge\frac{\ddbar\phi_{t}\wedge{\rm Re}\left(e^{-\sqrt{-1}\hat{\theta}}\Omega_{\phi}^{n-2}\right)}
{{\rm Re}\left(e^{-\sqrt{-1}\hat{\theta}}\Omega_{\phi}^{n}\right)} \\
& -\frac{n(n-1)}{2}(\sqrt{-1}\de\phi_{t}\wedge\dbar\eta+\sqrt{-1}\de\eta\wedge\dbar\phi_{t})
\wedge\frac{\ddbar\phi_{s}\wedge{\rm Re}\left(e^{-\sqrt{-1}\hat{\theta}}\Omega_{\phi}^{n-2}\right)}
{{\rm Re}\left(e^{-\sqrt{-1}\hat{\theta}}\Omega_{\phi}^{n}\right)} \\
& +Q(\nabla\phi_{s},\nabla\eta)
\frac{n\ddbar\phi_{t}\wedge{\rm Im}\left(e^{-\sqrt{-1}\hat{\theta}}\Omega_{\phi}^{n-1}\right)}
{{\rm Re}\left(e^{-\sqrt{-1}\hat{\theta}}\Omega_{\phi}^{n}\right)} \\
& -Q(\nabla\phi_{t},\nabla\eta)
\frac{n\ddbar\phi_{s}\wedge{\rm Im}\left(e^{-\sqrt{-1}\hat{\theta}}\Omega_{\phi}^{n-1}\right)}
{{\rm Re}\left(e^{-\sqrt{-1}\hat{\theta}}\Omega_{\phi}^{n}\right)} \\
& +Q(\nabla\phi_{t},\nabla Q(\nabla\phi_{s},\nabla\eta))
-Q(\nabla\phi_{s},\nabla Q(\nabla\phi_{t},\nabla\eta)).
\end{split}
\end{equation}

\subsection{Sectional Curvature}
In this subsection, we consider the sectional curvature:
\[
K(\phi_{t},\phi_{s}) = \frac{\langle R(\phi_{t},\phi_{s})\phi_{s},\phi_{t}\rangle}
{\langle\phi_{t},\phi_{t}\rangle \langle\phi_{s},\phi_{s}\rangle-\langle\phi_{t},\phi_{s}\rangle^{2}}.
\]
We focus on the numerator $\langle R(\phi_{t},\phi_{s})\phi_{s},\phi_{t}\rangle$. By (\ref{Expression of R 1}), we compute
\[
\begin{split}
& \langle R(\phi_{t},\phi_{s})\phi_{s},\phi_{t}\rangle \\
= {} & \frac{n(n-1)}{2}\int_{X}\phi_{t}\left[\sqrt{-1}\de\phi_{s}\wedge\dbar\phi_{s}+\sqrt{-1}\de\phi_{s}\wedge\dbar\phi_{s}\right]
\wedge\ddbar\phi_{t}\wedge{\rm Re}(e^{\sqrt{-1}\hat{\theta}}\Omega_{\phi}^{n-2}) \\
& -\frac{n(n-1)}{2}\int_{X}\phi_{t}\left[\sqrt{-1}\de\phi_{t}\wedge\dbar\phi_{s}+\sqrt{-1}\de\phi_{s}\wedge\dbar\phi_{t}\right]
\wedge\ddbar\phi_{s}\wedge{\rm Re}(e^{\sqrt{-1}\hat{\theta}}\Omega_{\phi}^{n-2}) \\
& +n\int_{X}\phi_{t}Q(\nabla\phi_{s},\nabla\phi_{s})\ddbar\phi_{t}\wedge{\rm Im}(e^{\sqrt{-1}\hat{\theta}}\Omega_{\phi}^{n-1}) \\
& -n\int_{X}\phi_{t}Q(\nabla\phi_{t},\nabla\phi_{s})\ddbar\phi_{s}\wedge{\rm Im}(e^{\sqrt{-1}\hat{\theta}}\Omega_{\phi}^{n-1}) \\
& +\frac{n}{2}\int_{X}\phi_{t}\left[\sqrt{-1}\de\phi_{t}\wedge\dbar Q(\nabla\phi_{s},\nabla\phi_{s})
+\sqrt{-1}\de Q(\nabla\phi_{s},\nabla\phi_{s})\wedge\dbar\phi_{t}\right]\wedge{\rm Im}(e^{\sqrt{-1}\hat{\theta}}\Omega_{\phi}^{n-1}) \\
& -\frac{n}{2}\int_{X}\phi_{t}\left[\sqrt{-1}\de\phi_{s}\wedge\dbar Q(\nabla\phi_{t},\nabla\phi_{s})
+\sqrt{-1}\de Q(\nabla\phi_{t},\nabla\phi_{s})\wedge\dbar\phi_{s}\right]\wedge{\rm Im}(e^{\sqrt{-1}\hat{\theta}}\Omega_{\phi}^{n-1}),
\end{split}
\]
where we used the definition of $Q$ for the last two terms. For convenience, we denote the terms on the right hand side by $I_{i}$ for $i=1,2,3,4,5,6$.

For $I_{1}$ and $I_{2}$, we have
\[
\begin{split}
\frac{2I_{1}}{n(n-1)}
= {} & \int_{X}\phi_{t}\sqrt{-1}\de\phi_{s}\wedge\dbar\phi_{s}\wedge\ddbar\phi_{t}\wedge{\rm Re}(e^{\sqrt{-1}\hat{\theta}}\Omega_{\phi}^{n-2}) \\
& +\int_{X}\phi_{t}\sqrt{-1}\de\phi_{s}\wedge\dbar\phi_{s}\wedge\ddbar\phi_{t}\wedge{\rm Re}(e^{\sqrt{-1}\hat{\theta}}\Omega_{\phi}^{n-2}) \\
= {} & (\sqrt{-1})^{2}\int_{X}\dbar(\phi_{t}\de\phi_{s}\wedge\dbar\phi_{s})\wedge\de\phi_{t}\wedge
{\rm Re}(e^{\sqrt{-1}\hat{\theta}}\Omega_{\phi}^{n-2}) \\
& -(\sqrt{-1})^{2}\int_{X}\de(\phi_{t}\de\phi_{s}\wedge\dbar\phi_{s})\wedge\dbar\phi_{t}\wedge
{\rm Re}(e^{\sqrt{-1}\hat{\theta}}\Omega_{\phi}^{n-2}) \\
= {} & -\int_{X}\sqrt{-1}\de\phi_{t}\wedge\dbar\phi_{t}\wedge\sqrt{-1}\de\phi_{s}\wedge\dbar\phi_{s}\wedge
{\rm Re}(e^{\sqrt{-1}\hat{\theta}}\Omega_{\phi}^{n-2}) \\
& +\int_{X}\sqrt{-1}\de\phi_{t}\wedge\dbar\phi_{s}\wedge\ddbar\phi_{s}\wedge{\rm Re}(e^{\sqrt{-1}\hat{\theta}}\Omega_{\phi}^{n-2}) \\
& -\int_{X}\sqrt{-1}\de\phi_{t}\wedge\dbar\phi_{t}\wedge\sqrt{-1}\de\phi_{s}\wedge\dbar\phi_{s}\wedge
{\rm Re}(e^{\sqrt{-1}\hat{\theta}}\Omega_{\phi}^{n-2}) \\
& +\int_{X}\sqrt{-1}\de\phi_{s}\wedge\dbar\phi_{t}\wedge\ddbar\phi_{s}\wedge{\rm Re}(e^{\sqrt{-1}\hat{\theta}}\Omega_{\phi}^{n-2}),
\end{split}
\]
which implies
\[
I_{1}+I_{2}
= {} -n(n-1)\int_{X}\sqrt{-1}\de\phi_{t}\wedge\dbar\phi_{t}\wedge\sqrt{-1}\de\phi_{s}\wedge\dbar\phi_{s}\wedge{\rm Re}(e^{\sqrt{-1}\hat{\theta}}\Omega_{\phi}^{n-2}).
\]
For $I_{3}$ and $I_{5}$, we have
\[
\begin{split}
I_{3} = {} & n\int_{X}\phi_{t}Q(\nabla\phi_{s},\nabla\phi_{s})\ddbar\phi_{t}\wedge
{\rm Im}(e^{\sqrt{-1}\hat{\theta}}\Omega_{\phi}^{n-1}) \\
= {} & \frac{n}{2}\int_{X}\phi_{t}Q(\nabla\phi_{s},\nabla\phi_{s})\ddbar\phi_{t}\wedge
{\rm Im}(e^{\sqrt{-1}\hat{\theta}}\Omega_{\phi}^{n-1}) \\
& -\frac{n}{2}\int_{X}\phi_{t}Q(\nabla\phi_{s},\nabla\phi_{s})\sqrt{-1}~\dbar\de\phi_{t}\wedge
{\rm Im}(e^{\sqrt{-1}\hat{\theta}}\Omega_{\phi}^{n-1}) \\
= {} & -\frac{n}{2}\int_{X}Q(\nabla\phi_{s},\nabla\phi_{s})\sqrt{-1}\de\phi_{t} \wedge\dbar\phi_{t}\wedge{\rm Im}(e^{\sqrt{-1}\hat{\theta}}\Omega_{\phi}^{n-1}) \\
& -\frac{n}{2}\int_{X}\phi_{t}\sqrt{-1}\de Q(\nabla\phi_{s},\nabla\phi_{s})\wedge\dbar\phi_{t}\wedge
{\rm Im}(e^{\sqrt{-1}\hat{\theta}}\Omega_{\phi}^{n-1}) \\
& -\frac{n}{2}\int_{X}Q(\nabla\phi_{s},\nabla\phi_{s})\sqrt{-1}\de\phi_{t} \wedge\dbar\phi_{t}\wedge
{\rm Im}(e^{\sqrt{-1}\hat{\theta}}\Omega_{\phi}^{n-1}) \\
& -\frac{n}{2}\int_{X}\phi_{t}\sqrt{-1}\de\phi_{t}\wedge\dbar Q(\nabla\phi_{s},\nabla\phi_{s})\wedge{\rm Im}(e^{\sqrt{-1}\hat{\theta}}\Omega_{\phi}^{n-1}),
\end{split}
\]
which implies
\[
\begin{split}
I_{3}+I_{5}
= {} & -n\int_{X}Q(\nabla\phi_{s},\nabla\phi_{s})\sqrt{-1}\de\phi_{t} \wedge\dbar\phi_{t}\wedge{\rm Im}(e^{\sqrt{-1}\hat{\theta}}\Omega_{\phi}^{n-1}) \\
= {} & -\int_{X}Q(\nabla\phi_{s},\nabla\phi_{s})Q(\nabla\phi_{t},\nabla\phi_{t}){\rm Re}(e^{\sqrt{-1}\hat{\theta}}\Omega_{\phi}^{n}),
\end{split}
\]
where we used the definition of $Q$:
\[
Q(\nabla\phi_{t},\nabla\phi_{t}) :=
\frac{n}{2}
\frac{
(\sqrt{-1}\de\phi_{t}\wedge\dbar\phi_{t}+\sqrt{-1}\de\phi_{t}\wedge\dbar\phi_{t})
\wedge{\rm Im}(e^{\sqrt{-1}\hat{\theta}}\Omega_{\phi}^{n-1})}{{\rm Re}(e^{\sqrt{-1}\hat{\theta}}\Omega_{\phi}^{n})}.
\]
Similarly, we have
\[
I_{4}+I_{6}
= \int_{X}Q(\nabla\phi_{t},\nabla\phi_{s})Q(\nabla\phi_{t},\nabla\phi_{s}){\rm Re}(e^{\sqrt{-1}\hat{\theta}}\Omega_{\phi}^{n}).
\]
Therefore, we obtain the expression of $\langle R(\phi_{t},\phi_{s})\phi_{s},\phi_{t}\rangle$:
\begin{equation}\label{Expression of R 2}
\begin{split}
& \langle R(\phi_{t},\phi_{s})\phi_{s},\phi_{t}\rangle \\
= {} & -n(n-1)\int_{X}\sqrt{-1}\de\phi_{t}\wedge\dbar\phi_{t}\wedge\sqrt{-1}\de\phi_{s}\wedge\dbar\phi_{s}\wedge{\rm Re}
\left(e^{-\sqrt{-1}\hat{\theta}}\Omega_{\phi}^{n-2}\right) \\
& -\int_{X}Q(\nabla\phi_{s},\nabla\phi_{s})Q(\nabla\phi_{t},\nabla\phi_{t}){\rm Re}\left(e^{-\sqrt{-1}\hat{\theta}}\Omega_{\phi}^{n}\right) \\
& +\int_{X}Q(\nabla\phi_{s},\nabla\phi_{t})Q(\nabla\phi_{s},\nabla\phi_{t}){\rm Re}\left(e^{-\sqrt{-1}\hat{\theta}}\Omega_{\phi}^{n}\right),
\end{split}
\end{equation}
where the definition of $Q$ is
\[
Q(\nabla\psi,\nabla\eta) :=
\frac{n}{2}
\frac{
(\sqrt{-1}\de\psi\wedge\dbar\eta+\sqrt{-1}\de\eta\wedge\dbar\psi)
\wedge{\rm Im}\left(e^{-\sqrt{-1}\hat{\theta}}\Omega_{\phi}^{n-1}\right)}
{{\rm Re}\left(e^{-\sqrt{-1}\hat{\theta}}\Omega_{\phi}^{n}\right)}.
\]

\subsection{Non-positivity of sectional curvature}
We claim that the sectional curvature is non-positive. For any $\phi\in\mathcal{H}$ and $\psi,\eta\in C^{\infty}(X)$, we define $\phi(t,s)=\phi+t\psi+s\eta$. Then $\phi_{t}=\psi$ and $\phi_{s}=\eta$. It suffices to show that
\[
R(\psi,\eta,\eta,\psi) = \langle R(\phi_{t},\phi_{s})\phi_{s},\phi_{t}\rangle
\]
is non-positive.  We will write each expression appearing on the right hand side of ~\eqref{Expression of R 2}) in local coordinates.  Fix a point $p$ and choose
coordinates so that $\omega_{i\bar{j}} = \delta_{i\bar{j}}$ and $(\alpha_{\phi})_{i\bar{j}} = \lambda_{i} \delta_{i\bar{j}}$.  We will use the notation
\[
r = \prod_{i=1}^{n} \sqrt{1+\lambda_i^2}, \qquad \Theta = \sum_{i=1}^{n}\arctan(\lambda_i),
\]
and write
\[
\widehat{(i_1\cdots i_{\ell})} = \bigwedge_{k\ne i_1,\ldots,i_{\ell}} \sqrt{-1}dz_k\wedge d\bar{z}_k.
\]
We begin with the first term of (\ref{Expression of R 2}). Expanding yields
\[
{\rm Re} \left(e^{-\sqrt{-1}\hat{\theta}}(\omega+\sqrt{-1}\alpha_{\phi})^{n-2}\right) = (n-2)!{\rm Re} \left(e^{-\sqrt{-1}\hat{\theta}}\sum_{i< j} \prod_{k\ne i,j} (1+\sqrt{-1}\lambda_k)\right) \widehat{(ij)}.
\]
Now write
\[
\begin{aligned}
\prod_{k\ne i,j} (1+\sqrt{-1}\lambda_k) &= \left(\prod_{k} (1+\sqrt{-1}\lambda_k)\right) \frac{(1-\sqrt{-1}\lambda_i)(1-\sqrt{-1}\lambda_j)}{(1+\lambda_i^2)(1+\lambda_j^2)} \\
&= re^{\sqrt{-1}\Theta} \frac{(1-\sqrt{-1}\lambda_i)(1-\sqrt{-1}\lambda_j)}{(1+\lambda_i^2)(1+\lambda_j^2)}.
\end{aligned}
\]
Therefore
\[
\begin{aligned}
&\left(e^{-\sqrt{-1}\hat{\theta}}\sum_{i< j} \prod_{k\ne i,j} (1+\sqrt{-1}\lambda_k)\right) \\
&= (r\cos(\Theta-\hat{\theta}) + \sqrt{-1}r \sin(\Theta-\hat{\theta}))\left(\sum_{i< j}\frac{1-\lambda_i\lambda_j-\sqrt{-1}(\lambda_i+\lambda_j)}{(1+\lambda_i^2)(1+\lambda_j^2)}\right).
\end{aligned}
\]
Taking the real part of this expression yields
\[
\begin{aligned}
{\rm Re} \left(e^{-\sqrt{-1}\hat{\theta}}(\omega+\sqrt{-1}\alpha_{\phi})^{n-2}\right)  &= (n-2)! r\cos(\Theta-\hat{\theta})\sum_{i< j} \frac{(1-\lambda_i\lambda_j)}{(1+\lambda_i^2)(1+\lambda_j^2)}\widehat{(ij)}\\
& +(n-2)!r\sin(\Theta-\hat{\theta})\sum_{i< j} \frac{(\lambda_i+\lambda_j)}{(1+\lambda_i^2)(1+\lambda_j^2)}\widehat{(ij)}.\\
\end{aligned}
\]
On the other hand, the coefficient of $ \sqrt{-1}dz_i\wedge d\bar{z}_i\wedge \sqrt{-1}dz_j\wedge d\bar{z}_j$ appearing in $\sqrt{-1}\de \psi \wedge\dbar \psi \wedge \sqrt{-1}\de \eta \wedge \dbar \eta$ is given by
\[
\begin{aligned}
&|\de_i\psi|^2|\de_j \eta|^2 + |\de_j\psi|^2|\de_i \eta|^2 -\de_j\psi \de_{\bar{j}}\eta \de_i\eta\de_{\bar{i}}\psi  -\de_i\psi \de_{\bar{i}}\eta \de_j\eta\de_{\bar{j}}\psi\\
&= |\de_i\psi|^2|\de_j \eta|^2 + |\de_j\psi|^2|\de_i \eta|^2-{\rm Re}(\de_j\psi \de_{\bar{j}}\eta \de_i\eta\de_{\bar{i}}\psi)  -{\rm Re}(\de_i\psi \de_{\bar{i}}\eta \de_j\eta\de_{\bar{j}}\psi).
\end{aligned}
\]

Thus, after dividing by ${\rm Re}(e^{-\sqrt{-1}\hat{\theta}}(\omega+\sqrt{-1}\alpha)^n)$ we arrive at an expression for the first term, suppressing the minus sign, and cancelling common factors of $n!$
\begin{equation}\label{eq: curvTerm1}
\begin{aligned}
&\frac{n(n-1)\sqrt{-1}\de\psi\wedge\dbar\psi\wedge\sqrt{-1}\de\eta\wedge\dbar\eta\wedge{\rm Re} \left(e^{-\sqrt{-1}\hat{\theta}}(\omega+\sqrt{-1}\alpha_{\phi})^{n-2}\right)}{{\rm Re}(e^{-\sqrt{-1}\hat{\theta}}(\omega+\sqrt{-1}\alpha)^n)} \\
&= \sum_{i\ne j} \frac{(1-\lambda_i\lambda_j) |\de_i\psi|^2|\de_j \eta|^2 }{(1+\lambda_i^2)(1+\lambda_j^2)}- \sum_{i\ne j} \frac{(1-\lambda_i\lambda_j){\rm Re}(\de_j\psi \de_{\bar{j}}\eta \de_i\eta\de_{\bar{i}}\psi )}{(1+\lambda_i^2)(1+\lambda_j^2)}\\
&\quad + \tan(\Theta-\hat{\theta})\sum_{i\ne j} \frac{(\lambda_i+\lambda_j)|\de_i\psi|^2|\de_j \eta|^2 }{(1+\lambda_i^2)(1+\lambda_j^2)}\\
&\quad - \tan(\Theta-\hat{\theta})\sum_{i\ne j} \frac{(\lambda_i+\lambda_j){\rm Re}(\de_j\psi \de_{\bar{j}}\eta \de_i\eta\de_{\bar{i}}\psi ) }{(1+\lambda_i^2)(1+\lambda_j^2)}.
\end{aligned}
\end{equation}
Before proceeding we will simplify this expression.  First observe that we can extend all sums over $i=j$, since the new terms cancel exactly in the top row, and the bottom row.  Next we observe that
\[
(\tan(\Theta-\hat{\theta})-\lambda_i)(\tan(\Theta-\hat{\theta})-\lambda_j) = \tan(\Theta-\hat{\theta})^2 -(\lambda_i+\lambda_j)\tan(\Theta-\hat{\theta})  + \lambda_i\lambda_j.
\]
Therefore, we can write the expression on the right hand side of~\eqref{eq: curvTerm1} as
\begin{equation}\label{eq: firstTerm}
\begin{aligned}
&(1+(\tan(\Theta-\hat{\theta}))^2) \sum_{ij}\frac{ |\de_i\psi|^2|\de_j \eta|^2 }{(1+\lambda_i^2)(1+\lambda_j^2)}\\
& - (1+(\tan(\Theta-\hat{\theta}))^2) \sum_{i,j}\frac{{\rm Re}(\de_j\psi \de_{\bar{j}}\eta \de_i\eta\de_{\bar{i}}\psi )}{(1+\lambda_i^2)(1+\lambda_j^2)}\\
&- \sum_{i,j}\frac{(\tan(\Theta-\hat{\theta}) - \lambda_i)(\tan(\Theta-\hat{\theta}) - \lambda_j)|\de_i\psi|^2|\de_j\eta|^2}{(1+\lambda_i^2)(1+\lambda_j^2)}\\
&+  \sum_{i,j}\frac{(\tan(\Theta-\hat{\theta}) - \lambda_i)(\tan(\Theta-\hat{\theta}) - \lambda_j){\rm Re}(\de_j\psi \de_{\bar{j}}\eta \de_i\eta\de_{\bar{i}}\psi )}{(1+\lambda_i^2)(1+\lambda_j^2)}.
\end{aligned}
\end{equation}
Next we consider the term $Q(\nabla \psi, \nabla \eta)$ appearing in~\eqref{Expression of R 2}.  Again we expand in coordinates
\[
\begin{aligned}
{\rm Im}\left(e^{-\sqrt{-1}\hat{\theta}}(\omega+\sqrt{-1}\alpha_{\phi})^{n-1}\right) &= (n-1)!{\rm Im} \left( e^{-\sqrt{-1}\hat{\theta}}\sum_{i} \prod_{k\ne i} (1+\sqrt{-1}\lambda_k)\right) \widehat{(i)} \\
&=(n-1)!{\rm Im} \left( re^{\sqrt{-1}(\Theta-\hat{\theta})}\sum_i\frac{(1-\sqrt{-1}\lambda_i)}{1+\lambda_i^2}\right) \widehat{(i)}\\
&= (n-1)!\sum_i \frac{(r\sin(\Theta-\hat{\theta}) - \lambda_ir\cos(\Theta-\hat{\theta}))}{1+\lambda_i^2} \widehat{(i)}.
\end{aligned}
\]
The $\sqrt{-1}dz_i\wedge d\bar{z}_i$ component of $(\sqrt{-1}\de\psi\wedge\dbar\eta+\sqrt{-1}\de\eta\wedge\dbar\psi)$ is given by
\[
\de_i\psi\de_{\bar{i}}\eta + \de_i\eta\de_{\bar{i}}\psi = 2{\rm Re}(\de_i\psi \de_{\bar{i}}\eta).
\]
Therefore
\[
Q(\nabla \psi, \nabla \eta) = \sum_i \frac{(\tan(\Theta-\hat{\theta}) - \lambda_i){\rm Re}(\de_i\psi\de_{\bar{i}}\eta )}{1+\lambda_i^2} .
\]
From this we obtain 
\begin{equation}\label{eq: secondTerm}
\begin{aligned}
&Q(\nabla \psi, \nabla \psi)Q(\nabla \eta, \nabla \eta) \\
&= \left(\sum_{i}\frac{(\tan(\Theta-\hat{\theta}) - \lambda_i)|\de_i\psi|^2}{1+\lambda_i^2}\right) \left(\sum_{j}\frac{(\tan(\Theta-\hat{\theta}) - \lambda_j)|\de_j\eta|^2}{1+\lambda_j^2}\right)\\
&= \sum_{i,j}\frac{(\tan(\Theta-\hat{\theta}) - \lambda_i)(\tan(\Theta-\hat{\theta})-\lambda_j)|\de_i\psi|^2|\de_j\eta|^2}{(1+\lambda_i^2)(1+\lambda_j^2)}.\\
\end{aligned}
\end{equation}
Similarly we have
\begin{equation}\label{eq: thirdTerm}
\begin{aligned}
Q(\nabla \psi, \nabla \eta)^2& = \sum_{i,j}\frac{(\tan(\Theta-\hat{\theta}) - \lambda_i)(\tan(\Theta-\hat{\theta})-\lambda_j){\rm Re}(\de_i\psi\de_{\bar{i}}\eta){\rm Re}(\de_j\psi\de_{\bar{j}}\eta ) }{(1+\lambda_i^2)(1+\lambda_j^2)}.\\
\end{aligned}
\end{equation}
Suppressing the integration and the volume form, we need to estimate~$-$\eqref{eq: firstTerm}~$-$\eqref{eq: secondTerm}~$+$~\eqref{eq: thirdTerm}. Note that~\eqref{eq: secondTerm} cancels exactly the term on the third line of~\eqref{eq: firstTerm}.  In order to proceed further, we note that
\begin{equation}\label{eq: redundant}
{\rm Re}(\de_j\psi \de_{\bar{j}}\eta \de_i\eta\de_{\bar{i}}\psi ) = {\rm Re}(\de_j\psi \de_{\bar{j}}\eta){\rm Re}(\de_i\psi \de_{\bar{i}}\eta) + {\rm Im}(\de_j\psi \de_{\bar{j}}\eta){\rm Im}(\de_i\psi \de_{\bar{i}}\eta).
\end{equation}
 Therefore, ~\eqref{eq: thirdTerm} can be used to cancel the term containing ${\rm Re}(\de_j\psi \de_{\bar{j}}\eta){\rm Re}(\de_i\psi \de_{\bar{i}}\eta)$) appearing on the third line of~\eqref{eq: firstTerm} after applying~\eqref{eq: redundant}.  Putting everything together we get the following expression for $-$\eqref{eq: firstTerm}~$-$\eqref{eq: secondTerm}~$+$~\eqref{eq: thirdTerm}
 \begin{equation}\label{eq: finCurv}
 \begin{aligned}
 &-(1+(\tan(\Theta-\hat{\theta}))^2) \sum_{i,j}\frac{ |\de_i\psi|^2|\de_j \eta|^2 }{(1+\lambda_i^2)(1+\lambda_j^2)}\\
 & +(1+(\tan(\Theta-\hat{\theta}))^2) \sum_{i,j}\frac{ {\rm Re}(\de_j\psi \de_{\bar{j}}\eta \de_i\eta\de_{\bar{i}}\psi )  }{(1+\lambda_i^2)(1+\lambda_j^2)}\\
&-  \sum_{i,j}\frac{(\tan(\Theta-\hat{\theta}) - \lambda_i)(\tan(\Theta-\hat{\theta}) - \lambda_j){\rm Im}(\de_j\psi \de_{\bar{j}}\eta){\rm Im}(\de_i\psi \de_{\bar{i}}\eta) }{(1+\lambda_i^2)(1+\lambda_j^2)}.
\end{aligned}
\end{equation}
The third line can be written as a square, and is therefore clearly negative.  We claim that the first line controls the second.  To do this we symmetrize the first sum to get
\[
(1+(\tan(\Theta-\hat{\theta}))^2) \sum_{ij}\frac{\frac{1}{2}( |\de_i\psi|^2|\de_j \eta|^2 + |\de_j\psi|^2|\de_i \eta|^2)}{(1+\lambda_i^2)(1+\lambda_j^2)}.
 \]
 Then it suffices to show that
 \[
 \frac{1}{2}( |\de_i\psi|^2|\de_j \eta|^2 + |\de_j\psi|^2|\de_i \eta|^2)-{\rm Re}(\de_j\psi \de_{\bar{j}}\eta \de_i\eta\de_{\bar{i}}\psi ) \geq0.
 \]
Write $X_{ji} = \de_j \psi\de_{i}\eta$.  Then we have
\[
\frac{1}{2}(|X_{ij}|^2 + |X_{ji}|^2) \geq |X_{ij}||X_{ji}| \geq {\rm Re}(X_{ji}\overline{X_{ij}})
\]
which is the desired inequality.  Let us now consider the equality case $R(\psi, \eta, \eta, \psi) =0$.  In the above notation we must have $|X_{ij}| = |X_{ji}|$ and $|X_{ij}||X_{ji}| ={\rm Re}(X_{ji}\overline{X_{ij}})$ and hence $X_{ij}=X_{ji}$.  From this it easily follows that $\de \psi \wedge \de \eta =0$, and hence $\de \psi, \de \eta$ are parallel, in the sense that for each $p\in X$ where $\de \psi \ne 0$ there is a number $c(p)\in \mathbb{C}$ such that $\de \psi = c(p) \de \eta$ (and vice versa whenever $\de \eta \ne 0$).  Finally, if $\de \psi, \de \eta$ are parallel, then the third term in~\eqref{eq: finCurv} becomes
\[
\left(\sum_{i}\frac{(\tan(\Theta-\hat{\theta}) - \lambda_i){\rm Im}(c)|\de_j\eta|^2}{(1+\lambda_i^2)}\right)^2
\]
and hence $R(\psi, \eta, \eta, \psi) =0$ if and only if $\de \psi, \de \eta$ are parallel and at each point $p\in X$ either $Q(\nabla \eta, \nabla \eta)=0$, or $c\in \mathbb{R}$.  It is rather easy to generate such ``flat 2-planes" in $T_{\phi}\mathcal{H}$.  For example, by taking $\eta = a\circ \psi +b$ where $a :\mathbb{R}\rightarrow \mathbb{R}$ is smooth.

\begin{remark}
The formula for the sectional curvature~\eqref{eq: finCurv} appears to be somewhat different from the formula obtained by Solomon for the curvature of the space of positive Lagrangians \cite{Solomon14}.  This is somewhat surprising given that the two spaces are related under mirror symmetry.  To understand their relation it is important to recall the real Fourier-Mukai transform of \cite{LYZ}; for this purpose we will restrict to the case of $[\alpha] \in H^{1,1}(X,\mathbb{Z})$, but the reader can check that everything we will say is true for general classes after appropriately including the $B$-field.  We will briefly recall this construction, but refer the reader to \cite{LYZ} and the references therein for a thorough treatment.  Recall that in semi-flat mirror symmetry a pair of mirror Calabi-Yau manifolds $X, \check{X}$ of real dimension $2n$ are given by the following construction.  There is a base $B$ of real dimension $n$, and a lattice $\Lambda$ so that
\[
\check{X} = T^{*}B/ \check{\Lambda}, \qquad X = TB/\Lambda.
\]
Here $\check{X}$ has a natural symplectic structure, and $X$ has a natural complex structure.  The latter is constructed in the usual way: if $(x_1,\ldots, x_n)$ are coordinates on the base $B$, and $(y_1,\ldots, y_n)$ are the natural coordinates on the fibers of $TB$, then the complex structure makes $x_i+\sqrt{-1}y_i$ holomorphic coordinates on $X$.  Under the real Fourier-Mukai transform a Lagrangian section of the SYZ fibration of $\check{X}$ is mapped to a holomorphic line bundle $L \rightarrow X$ with a hermitian metric $h$ such that $h(x,y) = h(x)$.  That is, $h$ does not depend on the fiber coordinates of the SYZ fibration on $X$.  Therefore the {\em true} SYZ transform of the space of positive Lagrangian sections of $\check{X}$ would be the space $\mathcal{H}_{inv}$ consisting of metrics on $L$ constant along the fibers of $X\rightarrow B$.  The tangent space to $\mathcal{H}_{inv}$ consists of real functions constant along the fibers, and hence for any $\psi \in T_{\phi}\mathcal{H}_{inv}$ we have (in the natural coordinates on $X$)
\[
\frac{\de}{\de z_i}\psi = \frac{1}{2}\frac{\de}{\de x_i}\psi = \frac{\de}{\de \bar{z}_i}\psi 
\]
With this restriction, one can then easily check that under the transformation in \cite{LYZ} the first two terms in~\eqref{eq: finCurv} yield exactly the formula in~\cite{Solomon14}, while the third term in~\eqref{eq: finCurv} vanishes identically.
\end{remark}

\section{Completion of $(\mathcal{H},d)$}\label{sec: completion}

In this section, we will show that the completion of $(\mathcal{H},d)$ is a $\mathrm{CAT}(0)$ space. Denote by $C(\mathcal{H})$ the set of Cauchy sequence $(\phi_i)$ in $\mathcal{H}$ and define a equivalence relation on $C(\mathcal{H})$ by
\[
(\phi_i)\sim (\varphi_i)\quad  \Longleftrightarrow \quad \lim_{i\rightarrow +\infty}d(\phi_i,\varphi_i)=0.
\]
Denote the completion of $\mathcal{H}$ by $\tilde{\mathcal{H}}=C(\mathcal{H})/\sim$ and define the metric on $\tilde{\mathcal{H}}$ by 
\[
\tilde d\left( [\phi_i],[\varphi_i]\right)=\lim_{i\rightarrow +\infty} d(\phi_i,\varphi_i).
\]

First let us isolate the main ingredient of the proof of Theorem \ref{CAT(0) space}.
\begin{proposition}\label{nonpositive-dist}
Let $P$, $Q$ and $R$ be three points in $\mathcal{H}$. For each $i\in \mathbb{N}$ let $\varphi_i(s),s\in [0,1]$ be the $i^{-1}$-geodesic from $P$ to $Q$. Then for all  $\lambda\in [0,1]$, we have
\begin{equation}
\limsup_{i\rightarrow +\infty} d\left(R,\varphi_i(\lambda)\right)^2\leq (1-\lambda)d(R,P)^2+\lambda d(R,Q)^2-\lambda(1-\lambda)d(P,Q)^2.
\end{equation}

\end{proposition}

We will follow closely the argument in \cite{CalabiChen2002}. To begin with, we first show that the convergence in classical sense will coincide with the convergence in metric $\tilde d$.{
\begin{lemma}\label{conv-coin}
Suppose $(\varphi_i)$ is a sequence in $\mathcal{H}$ such that $\varphi_i\rightarrow \varphi_\infty$ uniformly, then $(\varphi_i)$ is a Cauchy sequence with respect to metric $d$.
\end{lemma}
\begin{proof}
This follows from Lemma \ref{classical-metric-relation}, Theorem~\ref{Metric structure} and the definition directly.
\end{proof}
}

\begin{lemma}\label{2nd-der}
Let $P,Q,R$ be three points in $\mathcal{H}$. Suppose $\varphi(s,t),s,t\in [0,1]$ is a two parameter family of curves in $\mathcal{H}$ such that $\varphi(s,1)$ is a smooth curve from $P$ to $Q$ and for each $s_0\in [0,1]$, $\varphi(s_0,t)$ is a $\delta$-geodesic from $R=\varphi(s_0,0)$ to $\varphi(s_0,1)$ for some $\delta>0$. Then
$$\frac{\partial}{\partial t} \|\varphi_s\| \Big|_{t=1}\geq \|\varphi_s\|\Big|_{t=1}$$
\end{lemma}

\begin{proof}
Let $Y=\varphi_s$ and $X= \varphi_t$. Then
\begin{equation*}
\begin{split}
\frac{1}{2}\frac{\partial}{\partial t} \|Y\|^2&= \langle \nabla_X Y,Y\rangle= \langle \nabla_Y X,Y\rangle.
\end{split}
\end{equation*}

Then
\begin{equation*}
\begin{split}
\frac{1}{2} \frac{\partial}{\partial t} \|Y\|^2&=\langle \nabla_X \nabla_Y X,Y\rangle+\langle \nabla_Y X , \nabla_X Y\rangle\\
&=|\nabla_X Y|^2+\langle \nabla_Y \nabla_X X ,Y\rangle -K(X,Y)\\
&\geq |\nabla_X Y|^2+\langle \nabla_Y \nabla_X X ,Y\rangle.
\end{split}
\end{equation*}

Noted that from $\delta$-geodesic equation, we have
\begin{equation*}
\begin{split}
\nabla_X X&=\varphi_{tt}+Q(\nabla \varphi_t,\nabla\varphi_t)\\
&=\varphi_{tt}+n \frac{\sqrt{-1}\partial \varphi_t \wedge \bar\partial \varphi_t \wedge \Imnn}{\Ren}\\
&= -4\delta^2 e^{-2t} \frac{\Imn}{\Ren}.
\end{split}
\end{equation*}

Hence, (using the notation introduced in Definition~\ref{defn: notation})
\begin{equation*}
\begin{split}
&\quad \nabla_Y (\nabla_X X)\\
&=\partial_s \nabla_XX +Q(\nabla \varphi_s,\nabla \nabla_XX)\\
&=-4\delta^2 e^{-2t}\left[\frac{\partial}{\partial s}\frac{\Imn}{\Ren} \quad +Q\left(\nabla \frac{\Imn}{\Ren}, \nabla \varphi_s\right)\right]\\
&=-4\delta^2e^{-2t}\left(I+II\right).
\end{split}
\end{equation*}
where
\begin{equation*}
\begin{split}
I&=\frac{n\ddbar \varphi_s\wedge \mathrm{Re}\left(e^{-\sqrt{-1}\hat\theta} \Omega_\varphi^{n-1} \right)}{\Ren}\\
&\quad +\frac{\Imn}{\Ren} \frac{n\ddbar \varphi_s\wedge\mathrm{Im}\left(  e^{-\sqrt{-1}\hat\theta}\Omega_{\varphi}^{n-1}\right)}{\Ren}\\
&=III+IV.
\end{split}
\end{equation*}

In particular,
\begin{equation*}
\begin{split}
&\quad -4\delta^2e^{-2t}\langle III,Y\rangle\\
&=-4n\delta^2e^{-2t}\int_X \varphi_s{\ddbar \varphi_s\wedge \mathrm{Re}\left(e^{-\sqrt{-1}\hat\theta} \Omega_\varphi^{n-1} \right)}\\
&=4n\delta^2e^{-2t} \int_X \sqrt{-1} \partial \varphi_s\wedge \bar\partial \varphi_s \wedge\Renn
\end{split}
\end{equation*}

Write $F=\frac{\Imn}{\Ren} = \tan(\Theta_{\omega}(\alpha_{\phi}) - \hat{\theta})$ for convenience, then
\begin{equation*}
\begin{split}
&\quad -4\delta^2e^{-2t} \langle IV,Y\rangle\\
&= -4n\delta^2e^{-2t}  \int_X \varphi_sF {\ddbar \varphi_s\wedge\mathrm{Im}\left(  e^{-\sqrt{-1}\hat\theta}\Omega_{\phi}^{n-1}\right)}\\
&= 4n\delta^2 e^{-2t} \int_X \sqrt{-1} \partial\left( \varphi_s F \right)\wedge \bar\partial\varphi_s\wedge\mathrm{Im}\left(  e^{-\sqrt{-1}\hat\theta}\Omega_{\phi}^{n-1}\right)\\
&=4n\delta^2 e^{-2t} \int_X \sqrt{-1} \left( \varphi_s\partial F+ F \partial \varphi_s \right)\wedge \bar\partial\varphi_s\wedge\mathrm{Im}\left(  e^{-\sqrt{-1}\hat\theta}\Omega_{\phi}^{n-1}\right)\\
&= 4n\delta^2 e^{-2t} \int_XF \sqrt{-1} \partial\varphi_s \wedge \bar\partial \varphi_s \wedge \Imnn\\
&\quad +4n\delta^2e^{-2t} \int_X \varphi_s\sqrt{-1} \Re\left(\partial F\wedge \bar\partial\varphi_s \right) \wedge \Imnn
\end{split}
\end{equation*}

To summarize, we get
\begin{equation*}
\begin{split}
&\quad -4\delta^2e^{-2t}\langle I,Y\rangle\\
&=4n\delta^2e^{-2t} \int_X \sqrt{-1} \partial \varphi_s\wedge \bar\partial \varphi_s \wedge\Renn\\
&\quad +4n\delta^2 e^{-2t} \int_XF \sqrt{-1} \partial\varphi_s \wedge \bar\partial \varphi_s \wedge \Imnn\\
&\quad +4n\delta^2e^{-2t} \int_X \varphi_s\Re \left(\sqrt{-1}\partial F\wedge \bar\partial\varphi_s \right) \wedge \Imnn
\end{split}
\end{equation*}

For $II$:
\begin{equation*}
\begin{split}
&\quad -4\delta^2 e^{-2t}\langle II,Y\rangle\\
& =-4\delta^2 e^{-2t}\int_X\varphi_s L(\nabla F,\nabla \varphi_s)  \Ren \\
&=-4\delta^2 e^{-2t}\frac{n}{2}\int_X \varphi_s {\sqrt{-1}(\partial  F\wedge \bar\partial \varphi_s+\partial\varphi_s\wedge \bar\partial F)\wedge \Imnn}\\
&=-4n\delta^2 e^{-2t} \int_X \varphi_s  \Re\left(\sqrt{-1} \partial F\wedge \bar\partial \varphi_s \right)\wedge \Imnn
\end{split}
\end{equation*}

Therefore,
\begin{equation*}
\begin{split}
&\quad \langle \nabla_Y\nabla_XX,Y\rangle\\
&=4n\delta^2e^{-2t} \int_X \sqrt{-1} \partial \varphi_s\wedge \bar\partial \varphi_s \wedge\Renn\\
&\quad +4n\delta^2 e^{-2t} \int_XF \sqrt{-1} \partial\varphi_s \wedge \bar\partial \varphi_s \wedge \Imnn
\end{split}
\end{equation*}

To see the inequality, take a coordinate  at $p\in X$ such that $g_{i\bar j}=\delta_{ij}$ and $\a_\varphi=\lambda_i \delta_{ij}$. Then
\begin{equation*}
\begin{split}
&\quad \sqrt{-1}\partial \varphi_s \wedge \bar\partial\varphi_s \wedge \Renn\\
&= r(n-1)!\Re\left(\sum_{i=1}^n|\varphi_{si}|^2 \frac{e^{\sqrt{-1}(\Theta_{\omega}(\alpha_{\phi})-\hat\theta)}(1-\sqrt{-1}\lambda_i)}{1+\lambda_i^2}\right)\\
&=r(n-1)!\sum_{i=1}^n|\varphi_{si}|^2 \frac{ \cos(\Theta_{\omega}(\alpha_{\phi})-\hat\theta)+\lambda_i \sin(\Theta_{\omega}(\alpha_{\phi})-\hat\theta)}{1+\lambda_i^2}
\end{split}
\end{equation*}

On the other hand,
\begin{equation*}
\begin{split}
&\quad \frac{\Imn}{\Ren} \sqrt{-1} \partial\varphi_s \wedge \bar\partial \varphi_s \wedge \Imnn\\
&=r(n-1)!\tan (\Theta-\hat\theta) \Im\left( \sum_{i=1}^n|\varphi_{si}|^2\frac{e^{\sqrt{-1}(\Theta_{\omega}(\alpha_{\phi})-\hat\theta)}(1-\sqrt{-1}\lambda_i)}{1+\lambda_i^2}\right)\\
&=r(n-1)!\tan (\Theta-\hat\theta)\sum_{i=1}^n|\varphi_{si}|^2\frac{\sin(\Theta_{\omega}(\alpha_{\phi})-\hat\theta)-\lambda_i \cos(\Theta_{\omega}(\alpha_{\phi})-\hat\theta)}{1+\lambda_i^2}\\
&=r(n-1)!\sum_{i=1}^n|\varphi_{si}|^2\frac{\sin^2(\Theta_{\omega}(\alpha_{\phi})-\hat\theta)\sec(\Theta_{\omega}(\alpha_{\phi})-\hat\theta)-\lambda_i \sin(\Theta_{\omega}(\alpha_{\phi})-\hat\theta)}{1+\lambda_i^2}\\
\end{split}
\end{equation*}

Summing up, we have
\begin{equation*}
\begin{split}
 \langle \nabla_Y\nabla_XX,Y\rangle&\geq 0
\end{split}
\end{equation*}
since $\Theta-\hat\theta \in (-\frac{\pi}{2},\frac{\pi}{2})$ from the definition of $\mathcal{H}$. Therefore,
$\frac{1}{2}\frac{\partial^2}{\partial t^2}\|Y\|^2 \geq \|\nabla_{X}Y\|^2$ and hence from Cauchy inequality that
$$\frac{\partial^2}{\partial t^2}\|Y\|\geq 0.$$
Since $\varphi(s,0)\equiv R$, $Y(0)=0$, we have
$$\frac{\partial }{\partial t}\Big|_{t=1}\|Y\|\geq \|Y(s,1)\|.$$
This completes the proof.
\end{proof}

Now we are ready to prove Proposition \ref{nonpositive-dist}.
\begin{proof}
[Proof of Proposition \ref{nonpositive-dist}]
Let $\varphi^i(s),s\in [0,1]$ be a $i^{-1}$-geodesic from $P$ to $Q$. By estimates in \cite{CY18}, for all $j>i>>1$ there exist a two parameter family $\varphi_{i,j}(s,t)$, $s,t\in [0,1]$ such that $\varphi_{i,j}(s,1)=\varphi^i(s)$ and $\varphi_{i,j}(s_0,t),t\in [0,1]$ is a $j^{-1}$-geodesic from $R$ to $\varphi^i(s_0)$ for each $s_0\in [0,1]$. Let $ E^{tot}_{i,j}(s)$ be the total energy of the $j^{-1}$-geodesic $\varphi_{i,j}(s,t),t\in [0,1]$ connecting $R$ to $\varphi^{i}(s)$;  ie.
\[
\begin{aligned}
E^{tot}_{i,j}(s) := \int_{0}^{1} E(\varphi_{i,j})(s,t) dt &=  \int_{0}^{1} \| \de_t \varphi_{i,j}(s,t)\|^2 dt\\
&=\int^1_0 \int_X|\varphi_t|^2 \Ren dt.
\end{aligned}
\]
where $\| \cdot \|$ denotes the length with respect to the metric on $\mathcal{H}$.  To ease notation we will drop the index $i,j$ when the meaning is clear.  Furthermore, we denote $X(s,t)=\varphi_t(s,t)$ and $Y(s,t)=\varphi_s(s,t)$.

Our goal is to estimate from below $\frac{d^2}{ds^2}E^{tot}(s)$.  Compute
\begin{equation}
\begin{split}
\frac{1}{2} \frac{d  E^{tot}(s)}{ds}
&=\int^1_0 \langle \nabla_YX,X\rangle dt\\
&=\int^1_0 \langle \nabla_XY,X\rangle dt\\
&=\int^1_0 X\langle Y,X\rangle -\langle Y,\nabla_XX\rangle dt\\
&=\langle X,Y\rangle(s,t) |_{t=1}+4j^{-2} \int^1_0 \int_X e^{-2t}\varphi_s(s,t) \Imn  dt\\
&=: I(s)+j^{-2}II(s)
\end{split}
\end{equation}
where we have used  $\epsilon$-geodesic equation and that $\varphi(s,0)\equiv R$.

Next we compute,
\begin{equation}
\begin{split}
\frac{d}{ds} I=\frac{d}{d s} \langle X,Y\rangle\Big|_{t=1}
&=\langle \nabla_YX,Y\rangle\Big|_{t=1} + \langle X,\nabla_YY\rangle\Big|_{t=1} \\
&=\frac{1}{2} \frac{d}{dt}\Big|_{t=1} \|Y\|^2   + \langle X,\nabla_YY\rangle\Big|_{t=1} \\
&\geq  \|Y(s,1)\|^2+ \langle X,\nabla_YY\rangle\Big|_{t=1} \\
&= \|\phi^{i}_{s}(s)\|^2 + \langle X,\nabla_YY\rangle\Big|_{t=1},
\end{split}
\end{equation}
here  we have used Lemma \ref{2nd-der}. Since $\varphi(s,1), s\in [0,1]$ is a $i^{-1}$-geodesic we have
\begin{equation}
\begin{split}
 \langle X,\nabla_YY\rangle\Big|_{t=1} &=-4i^{-2} e^{-2}\int_X \varphi_t(s,1) \Imn.
\end{split}
\end{equation}
Consider the $1$-form on $\mathcal{H}$ given by
\[
T_{\phi} \mathcal{H} \ni \psi \longrightarrow \delta \mathcal{J}(\psi) :=  -\int_{X}\psi \Imn.
\]
It was shown in \cite{CY18} that, after fixing a base point in $\mathcal{H}$, $\delta \mathcal{J}$ integrates to a well-defined function $\mathcal{J}: \mathcal{H} \rightarrow \mathbb{R}$.  Furthermore, in \cite{CY18} it was shown that $\mathcal{J}$ is convex along $\epsilon$-geodesics.  Fix $P \in \mathcal{H}$ as the basepoint, for convenience.  Since $\phi(s,t)$ is a $j^{-1}$-geodesic in $t$, the convexity of $\mathcal{J}$ implies
\[
\begin{aligned}
-\int_X \varphi_t(s,1) \Imn &= \frac{d}{dt}\bigg|_{t=1}\mathcal{J}(\phi(s,t)) \\
&\geq \mathcal{J}(\phi(s,1)) - \mathcal{J}(\phi(s,0))\\
& = \mathcal{J}(\phi^{i}(s)) - \mathcal{J}(R).
\end{aligned}
\]
Now, since $\phi^{i}(s)$ is a $i^{-1}$-geodesic between $P,Q \in \mathcal{H}$, the uniform estimates in Theorem~\ref{thm: CY} imply
\[
|\mathcal{J}(\phi^{i}(s))| \leq C
\]
for a constant $C$ independent of $i$.  Thus we obtain
\[
 \langle X,\nabla_YY\rangle\Big|_{t=1} \geq -Ci^{-2}
 \]
 for a uniform constant $C$.

It remains to consider the term $II(s)$ above.  Since $\phi(s,t)$ is a $j^{-1}$ geodesic in $t$ for each fixed $s$, the function $\phi_s(s,t)$ solves the linearized $j^{-1}$-geodesic equation.  Thus, arguing in the same way as the proof of Lemma~\ref{triangle} we have
\[
\sup_{(x,t) \in X\times [0,1]} |\phi_{s}(s,t)| \leq \sup_{x\in X} |\phi^{i}_{s}(s)|.
\]
Thus, applying the $C^2$ estimates in Theorem~\ref{thm: CY},  we obtain that there is a constant $C_i$ depending on $i$ so that $| II(s)|\leq C_i$.  Consider the quantity
\[
F_{i,j}(s) = E^{tot}_{i,j}(s)-2j^{-2}\int_{0}^{s} II_{i,j}(\xi)d\xi-2\int_{0}^s\int_0^\xi \|\de_{\eta} \phi^{i}(\eta)\|^2 d\eta d\xi + Ci^{-2}s^2
\]
The above calculation shows that $F_{i,j}''(s) \geq 0$, and so $F_{i,j}(s) \leq s F_{i,j}(1+ (1-s)F_{i,j}(0)$.  We will now pass to the limit as $j\rightarrow \infty$, and then $i\rightarrow \infty$.  First, by the $j$-independent bounds for $II$, and Corollary~\ref{cor: constSpeed} we get that 
\[
\lim_{j\rightarrow \infty}F_{i,j}(s) \rightarrow  d(R, \phi^{i}(s))^2-2\int_{0}^s\int_0^\xi \|\de_{\eta} \phi^{i}(\eta)\|^2 d\eta d\xi + Ci^{-2}s^2.
\]
Next we take the limit as $i\rightarrow \infty$.  Only the second term needs to be understood.  Again, by Corollary~\ref{cor: constSpeed} we obtain
\[
\lim_{i\rightarrow \infty}\int_{0}^s\int_0^\xi \|\de_{\eta} \phi^{i}(\eta)\|^2 d\eta d\xi = d(P,Q)^2 \int_0^s \xi d\xi = \frac{s^2}{2}d(P,Q)^2
\]
Putting everything together we obtain
\[
\limsup_{i\rightarrow \infty} d(R, \phi^i(s)) \leq s d(R,Q)^2 + (1-s)d(R,P)^2 -s(s-1)d(P,Q)^2
\]
which is the desired result.

\end{proof}
\begin{proposition}\label{geodesic-space-1}
Let $P, Q$ be two points in $\mathcal{H}$ and $\varphi_i(s),s\in[0,1]$ be the unique $i^{-1}$-geodesic from $P$ to $Q$. Then, for each $s\in [0,1]$,  $(\varphi_i(s))_{i\in \mathbb{N}}$ is an element in $\tilde{\mathcal{H}}$.  Furthermore, $\varphi(s)=[\varphi_i(s)],s\in[0,1]$ is a geodesic connecting $[P]$ and $[Q]$ in $\tilde{\mathcal{H}}$, and $\varphi(s)$ satisfies
\[
\tilde{d}(\varphi(s), \varphi(t)) = |t-s| \tilde{d}(P,Q)
\]
\end{proposition}
\begin{proof}

 For each $\lambda\in[0,1]$ and $k>j\in \mathbb{N}$, we apply Proposition \ref{nonpositive-dist} with $R=\varphi_j(\lambda)$,
\begin{equation}
\begin{split}
\limsup_{k\rightarrow +\infty}d(\varphi_j(\lambda),\varphi_{k}(\lambda))^2&\leq (1-\lambda) d(\varphi_j(\lambda),P)^2+\lambda d(\varphi_j(\lambda),Q)^2 \\
&\quad -\lambda(1-\lambda)d(P,Q)^2.
\end{split}
\end{equation}

Since $\varphi_j$ is a $j^{-1}$-geodesic from $P$ to $Q$, Corollary~\ref{cor: constSpeed} implies
\[
\lim_{j\rightarrow \infty}d(\varphi_j(\lambda),P)^2 = \lambda^2 d(P,Q)^2 .
\]
Similarly,
\[
\lim_{j\rightarrow \infty}d(\varphi_j(\lambda),Q)^2 =(1-\lambda)^2 d(P,Q)^2 .
\]
In fact, for all $0\leq s\leq t\leq 1$,
\begin{equation}
\lim_{j\rightarrow \infty}d(\varphi_j(s),\varphi_j(t))^2 = (s-t)^2d(P,Q)^2 .
\end{equation}

Therefore, for any $\lambda\in [0,1]$,
\begin{equation}
\limsup_{j\rightarrow +\infty}\limsup_{k\rightarrow +\infty}d(\varphi_j(\lambda),\varphi_{k}(\lambda))^2=0.
\end{equation}
Together with pre-compactness following from Theorem~\ref{thm: CY} and Lemma \ref{conv-coin}, we have $(\varphi_i (\cdot))\in\tilde{\mathcal{H}}$. On the other hand, using above estimates and triangle inequality,
\begin{equation}
\begin{split}\tilde d([P],[Q])&\leq \limsup_{j\rightarrow \infty}d(\varphi_j(s),P)+\limsup_{j\rightarrow \infty}d(\varphi_j(s),\varphi_j(t))+\limsup_{j\rightarrow \infty}d(\varphi_j(t),Q)\\
&\leq d(P,Q).
\end{split}
\end{equation}

Since $P,Q\in \mathcal{H}$, it follows that $d(P,Q)=\tilde d([P],[Q])$ and thus
\begin{equation}\label{geodesic-distance-esti}
\tilde d([\varphi_i(s)],[\varphi_i(t)])=|s-t| d(P,Q)
\end{equation}
for any $0\leq s\leq t\leq 1$. Thus, $[\varphi_i(\cdot)]$ is a geodesic from $[P]$ to $[Q]$.
\end{proof}

\begin{corollary}\label{geodesic-space-2}
$(\tilde{\mathcal{H}},\tilde d)$ is a $\mathrm{CAT}(0)$ space.
\end{corollary}
\begin{proof}
We first show that $(\tilde{\mathcal{H}},\tilde d)$ is a geodesic metric space. Let $P$ and $Q$ be two points in $\tilde{\mathcal{H}}$ where $P=[P_i]$ and $Q=[Q_i]$ are represented by two Cauchy sequences in $(\mathcal{H},d)$. We will also regard each $P_i,Q_i$ as elements in $\tilde{\mathcal{H}}$. By Proposition \ref{geodesic-space-1}, for each $i$ we can find a geodesic $\varphi_{i}(s),s\in [0,1]$ in $\tilde{\mathcal{H}}$ connecting $P_i$ to $Q_i$. We first claim that $\varphi_i(s)$ is Cauchy with respect to $\tilde d$ for each $s\in [0,1]$. By applying Proposition \ref{nonpositive-dist} together with Proposition \ref{geodesic-space-1} and \eqref{geodesic-distance-esti}, for each $i,j$,
\begin{equation}
\begin{split}
\tilde d(\varphi_i(\lambda),\varphi_j(\lambda))^2&\leq (1-\lambda)\tilde d(\varphi_j(\lambda),P_i)^2+\lambda \tilde d(\varphi_j(\lambda),Q_i)^2\\
&\quad -\lambda(1-\lambda)\tilde d(P_i,Q_i)^2\\
&\leq (1-\lambda)\left( \lambda \tilde d(P_j,Q_j)+\tilde d(P_j,P_i)\right)^2\\
&\quad +\lambda\left( (1-\lambda )\tilde d(P_j,Q_j)+\tilde d(Q_j,Q_i)\right)^2\\
&\quad -\lambda(1-\lambda)\tilde d(P_i,Q_i)^2.
\end{split}
\end{equation}
The right hand side converges to $0$ as $i,j\rightarrow +\infty$ since $P_i$ and $Q_i$ are Cauchy sequence in $\mathcal{H}$ and hence $\tilde{\mathcal{H}}$. This shows that $\varphi_i(\cdot )\rightarrow \varphi(\cdot)$ on $[0,1]$ as $i\rightarrow +\infty$. In particular, $P$ and $Q$ can be connected by a curve $\varphi(s),s\in[0,1]$. Moreover, from \eqref{geodesic-distance-esti} we get that for $0\leq s\leq t\leq 1$, $\tilde d(\varphi_i(t),\varphi_i(s))=|s-t| \tilde d(P_i,Q_i)$ and hence $\tilde d(\varphi(t),\varphi(s))=|s-t| \tilde d(P,Q)$
for any $P,Q\in\tilde{\mathcal{H}}$. The structure of $\mathrm{CAT}(0)$ follows from the inequality inherited form Proposition \ref{nonpositive-dist}, see \cite{BridsonHaefliger1999} for example.
\end{proof}

\section{$C^{1,1}$ regularity of geodesics}\label{C11 section}
In this section, we obtain an improved regularity result for geodesics in the space $\mathcal{H}$.  Namely, we improve the regularity in Theorem~\ref{thm: CY} to full $C^{1,1}$ regularity.  Recall from Section~\ref{Geodesics section} that the $\epsilon$ geodesic equation can be written as PDE on $\mathcal{X}=X\times\mathcal{A}$.  Namely, consider the $\hat{\omega}_{\epsilon} := \pi^{*}\omega+\epsilon^2\sqrt{-1}dt\wedge d\ov{t}$.  Then the $\epsilon$ geodesic equation is
\[
\Theta_{\hat{\omega}_{\epsilon}}(\pi^{*}\alpha + \mn D\ov{D}\Phi) = \hat{\theta}.
\]
Since the metric $\hat{\omega}_{\epsilon}$ becomes degenerate when $\epsilon\rightarrow0$ it is more convenient to rescale.  Define
\[
\mathcal{X}_{\epsilon}=X\times\mathcal{A}_{\epsilon}, \quad \mathcal{A}_{\epsilon}=\{t\in\mathbb{C}~|~e^{-1}\epsilon\leq |t| \leq \epsilon\}.
\]
After rescaling $t\mapsto \epsilon t$, the $\epsilon$-geodesic equation becomes the deformed Hermitian-Yang-Mills equation on $\mathcal{X}_{\epsilon}$ with boundary data, and background metric $\hat{\omega} := \hat{\omega}_{1}$;
\[
\begin{cases}
{\rm Im}\left(e^{-\sqrt{-1}\hat{\theta}}\left(\pi^{*}\omega+dt\wedge d\ov{t}
+\sqrt{-1}(\pi^{*}\alpha+\sqrt{-1}D\ov{D}\Phi^{\epsilon})\right)^{n+1}\right) = 0, \\[2mm]
{\rm Re}\left(e^{-\sqrt{-1}\hat{\theta}}\left(\omega+\sqrt{-1}(\alpha+\ddbar\Phi^{\epsilon})\right)^{n}\right) > 0, \\[2mm]
\Phi^{\epsilon}(\cdot,1) = \vp_{0}, \ \Phi^{\epsilon}(\cdot,e^{-1}) = \vp_{1}.
\end{cases}
\]

In order to study the existence and regularity of the above equation, the second author and Yau \cite{CY18} considered the specified Lagrangian phase equation on $(\mathcal{X}_{\epsilon},\hat{\omega})$:
\begin{equation}\label{SLPE}
\begin{cases}
F(\ti{\alpha}_{\vp}) := \Theta_{\hat{\omega}}( \ti{\alpha}_{\vp})=\sum_{i=0}^{n}\arctan(\mu_{i}) = h, \\[1mm]
\vp(\cdot,t)|_{|t|=\epsilon} = \vp_{0}, \ \vp(\cdot,t)|_{t=e^{-1}\epsilon}= \vp_{1},
\end{cases}
\end{equation}
where $\ti{\alpha}_{\vp}=\pi^{*}\alpha+\sqrt{-1}D\ov{D}\vp$, $\mu_{i}$ are eigenvalues of $\ti{\alpha}_{\vp}$ with respect to $\hat{\omega}$, $h: \mathcal{X}_{\epsilon}\rightarrow\left((n-1)\frac{\pi}{2}+\eta,(n+1)\frac{\pi}{2}-\eta\right)$ is a $S^{1}$ invariant function on $\mathcal{X}_{\epsilon}$, and $\vp_{0},\vp_{1}\in\mathcal{H}$, and have introduced the notation $F$ for convenience (and to be consistent with \cite{CY18}).

The second author and Yau \cite{CY18} proved that there exists a constant $C$ independent of $\epsilon$ such that
\begin{equation}\label{Spatial estimate}
\sup_{\mathcal{X}_{\epsilon}}|\vp|+\sup_{\mathcal{X}_{\epsilon}}|\nabla^{X}\vp|+\sup_{\mathcal{X}_{\epsilon}}|\nabla^{X}\ov{\nabla^{X}}\vp| \leq C,
\end{equation}
\[
\sup_{\mathcal{X}_{\epsilon}}|\nabla^{X}\nabla_{\ov{t}}\vp|+\sup_{\mathcal{X}_{\epsilon}}|\nabla_{t}\vp| \leq C \epsilon^{-1},
\]
\[
\sup_{\mathcal{X}_{\epsilon}}|\nabla_{t}\nabla_{\ov{t}}\vp| \leq C \epsilon^{-2}.
\]

To prove Theorem \ref{C11 regularity}, it suffices to prove the following real Hessian estimate for $\vp$:

\begin{theorem}\label{C11 estimate}
Let $\vp$ solve the specified Lagrangian phase equation (\ref{SLPE}). Then there exists a constant $C$ depending only on $\sup_{\mathcal{X}_{\epsilon}}|\nabla^{X}\vp|$, $\sup_{\mathcal{X}_{\epsilon}}|\nabla^{X}\ov{\nabla^{X}}\vp|$,  $h$, $\alpha$ and $(X,\omega)$ such that
\[
\sup_{\mathcal{X}_{\epsilon}}|\nabla^{X}\nabla^{X}\vp| \leq C.
\]
\end{theorem}

\subsection{Some properties of Lagrangian operator}

For convenience, we denote $\ti{\alpha}_{\vp}$ by $\ti{\alpha}$, and use the following notations:
\[
F^{i\ov{j}} = \frac{\de F}{\de \ti{\alpha}_{i\ov{j}}}, \
F^{i\ov{j},k\ov{l}} = \frac{\de^{2} F}{\de \ti{\alpha}_{i\ov{j}}\de \ti{\alpha}_{k\ov{l}}}.
\]
For any point $x_{0}\in \mathcal{X}_{\epsilon}$, let $\{z_{i}\}_{i=0}^{n}$ be a local coordinate system centered at $x_{0}$ such that
\[
\hat{g}_{i\ov{j}} = \delta_{ij}, \ \ti{\alpha}_{i\ov{j}} = \delta_{ij}\mu_{i},
\ \mu_{0} \geq \mu_{1} \geq \cdots \geq \mu_{n} \ \text{at $x_{0}$}.
\]
Then at $x_{0}$, we have (see e.g. \cite{Gerhardt96,Spruck05})
\[
F^{i\ov{j}} = \frac{\delta_{ij}}{1+\mu_{i}^{2}}
\]
and
\[
F^{i\ov{j},k\ov{l}} =
\begin{cases}
F^{i\ov{j},j\ov{i}} \quad &  \mbox{if $i=l, k=j$};  \\
0   \quad &  \mbox{otherwise},
\end{cases}
\]
where
\[
F^{i\ov{j},j\ov{i}} = -\frac{\mu_{i}+\mu_{j}}{(1+\mu_{i}^{2})(1+\mu_{j}^{2})}.
\]

\begin{lemma}\label{Properties of Lagrangian}
Suppose that $\mu_{0}\geq\mu_{1}\geq\cdots\geq\mu_{n}$ satisfy
\[
\sum_{i=0}^{n}\arctan(\mu_{i}) \geq (n-1)\frac{\pi}{2}+\eta
\]
for some $\eta>0$. We have
\begin{enumerate}
\item $\mu_{0}\geq\mu_{1}\geq\cdots\geq\mu_{n-1}>0$ and $\mu_{n-1}+\mu_{n}\geq0$.

\vspace{3pt}

\item $\mu_{n-1}\geq\tan(\frac{\eta}{2})$ and $\mu_{n}\geq -\cot(\eta)$.

\vspace{3pt}

\item If $\mu_{n}<0$, then
\[
\mu_{n-1} \geq \tan(\eta_{1})  \ \text{and} \  \sum_{i=0}^{n}\frac{1}{\mu_{i}} < -\tan(\eta_{1}).
\]
\end{enumerate}
\end{lemma}

\begin{proof}
We refer the reader to \cite[Lemma 3.1]{CY18}.
\end{proof}

\subsection{Proof of Theorem \ref{C11 estimate}}

\begin{proof}[Proof of Theorem \ref{C11 estimate}]
We consider the following quantity:
\[
Q(p,t,\xi) = \log(\vp_{\xi\xi})+f(|\nabla^{X}\vp|^{2}),
\]
where $(p,t)\in\mathcal{X}_{\epsilon}$ and $\xi$ is a $g$-unit vector in $T_{p}X$ and
\[
f(s) = -\log\left(1+\sup_{\mathcal{X}_{\epsilon}}|\nabla^{X}\vp|^{2}-s\right).
\]
Let $(p_{0},t_{0},V)$ be the maximum point of $Q$. Near $p_{0}\in X$, we choose holomorphic normal coordinates $\{w_{i}\}_{i=1}^{n}$ for $(X,\omega)$ centered at $p_{0}$. We define $w_{0}=t-t_{0}$, then $\{w_{i}\}_{i=0}^{n}$ becomes a holomorphic coordinates for $(\mathcal{X}_{\epsilon},\hat{\omega})$ centered at $(p_{0},t_{0})$. For convenience, we denote $(p_{0},t_{0})$ by $x_{0}$. After making a linear change of coordinates, we obtain a new holomorphic coordinates $\{z_{i}\}_{i=0}^{n}$ such that
\[
\frac{\de\hat{g}_{i\ov{j}}}{\de z_{k}}=0, \ \hat{g}_{i\ov{j}}=\delta_{ij}, \ \ti{\alpha}_{i\ov{j}} = \delta_{ij}\mu_{i},
\ \mu_{0} \geq \mu_{1} \geq \cdots \geq \mu_{n} \ \ \text{at $x_{0}$},
\]
where $\ti{\alpha}_{i\ov{j}}=(\ti{\alpha}_{\vp})_{i\ov{j}}$.

We extend $V\in T_{p_{0}}X$ to be vector field near $x_{0}$ by taking the components to be constant. For convenience, we use the following notations:
\[
\de_{i} = \frac{\de}{\de z_{i}}, \ \de_{\ov{i}} = \frac{\de}{\de \ov{z}_{i}}, \ \ W_{i} = \frac{\de}{\de w_{i}}, \ W_{\ov{i}} = \frac{\de}{\de \ov{w}_{i}}
\]
and
\[
\ g_{VV}=g(V,V), \ g_{W_{i}W_{\ov{j}}} = g(W_{i},W_{\ov{j}}), g^{W_{i}W_{\ov{j}}} = (g_{W_{i}W_{\ov{j}}})^{-1}.
\]

We note that $W_{0}=\de_{t}$ is time vector field, and $V$, $W_{i}$ ($1\leq i\leq n$) are spatial vector fields. By the definitions of $\{w_{i}\}_{i=0}^{n}$ and $\{z_{i}\}_{i=0}^{n}$,  the components of $\de_{i}$ in the basis $\{W_{i}\}_{i=0}^{n}$ are constants and vice versa. We assume
\[
V = \sum_{j=1}^{n}(v_{j}W_{j}+\ov{v_{j}}W_{\ov{j}}), \
\de_{i} = \sum_{j=0}^{n}\rho_{ij}W_{j}, \
W_{i} = \sum_{j=0}^{n}\rho^{ij}\de_{j},
\]
where $v_{i}$ are constants, $(\rho_{ij})$ and $(\rho^{ij})$ is a constant unitary matrices, and $(\rho^{ij})$ is the inverse of $(\rho_{ij})$.

Near $x_{0}$, we define
\[
\hat{Q} = \log(g_{VV}^{-1}\vp_{VV})+f(|\nabla^{X}\vp|^{2})
\]
It is clear that $\hat{Q}$ achieves its maximum at $x_{0}$.

\begin{lemma}\label{Estimate 1}
At $x_{0}$, we have
\begin{enumerate}
\item $|W_{i}W_{\ov{j}}(\vp)| \leq C$ for $1\leq i,j\leq n$.

\vspace{3pt}

\item $|W_{0}W_{\ov{0}}(\vp)|=|W_{0}W_{0}(\vp)| \leq C\mu_{0}$.

\vspace{3pt}

\item $|W_{i}W_{\ov{0}}(\vp)|=|W_{i}W_{0}(\vp)| \leq C\sqrt{\mu_{0}}$ for $1\leq i\leq n$.
\end{enumerate}
\end{lemma}

\begin{proof}
(1) follows from the uniform estimate (\ref{C11 estimate}) for the spatial second order derivatives. For (2), recalling $\mu_{0}$ is the largest eigenvalue and combining this with Lemma \ref{Properties of Lagrangian} (2), we obtain
\[
\max_{0\leq i\leq n}|\mu_{i}| \leq C\mu_{0},
\]
which implies
\[
|W_{0}W_{\ov{0}}(\vp)| \leq C\mu_{0}.
\]
Since $\vp$ is $S^{1}$-invariant, so we have
\[
|W_{0}W_{0}(\vp)| = |W_{0}W_{\ov{0}}(\vp)| \leq C\mu_{0}.
\]

For (3), by Lemma \ref{Properties of Lagrangian} (2), we obtain $\ti{\alpha}\geq-C\hat{\omega}$. Then there exists a uniform constant $C$ such that $\vp_{i\ov{j}}+C\delta_{ij}$ is positive definite. Combining this with (1) and (2), we see that
\[
|W_{i}W_{\ov{0}}(\vp)|^{2} \leq  (W_{0}W_{\ov{0}}(\vp)+C)(W_{i}W_{\ov{i}}(\vp)+C) \leq C\mu_{0}.
\]
Since $\vp$ is $S^{1}$-invariant, we obtain
\[
|W_{i}W_{\ov{0}}(\vp)|=|W_{i}W_{0}(\vp)| \leq C\sqrt{\mu_{0}}.
\]
\end{proof}

\begin{lemma}\label{Estimate 2}
At $x_{0}$, there exists a uniform constant $C$ such that
\begin{enumerate}
\item $|\mu_{i}| \leq C$ for $1\leq i\leq n$.

\vspace{3pt}

\item $|\rho_{0i}|+|\rho_{i0}|\leq \frac{C}{\sqrt{\mu_{0}}}$ for $1\leq i\leq n$.
\item $|\rho^{0i}|+|\rho^{i0}|\leq \frac{C}{\sqrt{\mu_{0}}}$ for $1\leq i\leq n$.
\end{enumerate}
\end{lemma}

\begin{proof}
At $x_{0}$, since the first derivative of $\hat{g}$ is zero, we have
\[
|\ti{\alpha}_{W_{i}W_{\ov{j}}}| \leq |W_{i}W_{\ov{j}}(\vp)|+C, \quad \text{for $0\leq i,j \leq n$}.
\]
Combining this with Lemma \ref{Estimate 1},
\begin{equation}\label{Estimate 2 eqn 1}
\begin{split}
|\ti{\alpha}_{W_{i}W_{\ov{j}}}| \leq C \quad \text{for $1\leq i,j\leq n$},\\
|\ti{\alpha}_{W_{i}W_{\ov{0}}}|=|\ti{\alpha}_{W_{0}W_{\ov{i}}}| \leq C\sqrt{\mu_{0}}, \quad  \text{for $1\leq i\leq n$}.
\end{split}
\end{equation}

We first prove $(1)$. Since the trace of matrix is invariant under change of basis,
\begin{equation}\label{Estimate 2 eqn 2}
\mu_0 +\sum_{i=1}^n \mu_i = \ti{\alpha}_{W_{0}W_{\ov{0}}}+\sum_{i=1}^n \tilde \alpha_{W_i W_{\ov{j}}}.
\end{equation}
Using the fact that $\mu_0$ is the largest eigenvalue, we obtain
\[
\ti{\alpha}_{W_{0}W_{\ov{0}}} \leq \mu_{0}.
\]
Combining this with (\ref{Estimate 2 eqn 1}) and (\ref{Estimate 2 eqn 2}),
\begin{equation}\label{Estimate 2 eqn 3}
\sum_{i=1}^n \mu_i = \left(\ti{\alpha}_{W_{0}W_{\ov{0}}}-\mu_0\right)+\sum_{i=1}^n \tilde \alpha_{W_i W_{\ov{j}}}
\leq \sum_{i=1}^n \tilde \alpha_{W_i W_{\ov{j}}}
\leq C.
\end{equation}
Then (1) follows from (\ref{Estimate 2 eqn 3}) and Lemma \ref{Properties of Lagrangian} (2).

Since the matrix $(\rho_{ij})$ is unitary, (2) and (3) are equivalent. It suffices to prove (2). Without loss of generality, we assume that $\mu_{0}\geq1$. Using  (\ref{Estimate 2 eqn 1}) and  (\ref{Estimate 2 eqn 2}), we compute
\[
\begin{split}
\mu_0 = {} &  \tilde \alpha_{0\ov 0}
= \sum_{i,j=0}^{n}\rho_{0i}\ov \rho_{0j}\tilde \alpha_{W_i W_{\ov j}} \\
= {} & |\rho_{00}|^{2}\tilde \alpha_{W_0 W_{\ov 0}}+\sum_{i=1}^{n}\rho_{0i}\ov {\rho_{00}}\tilde \alpha_{W_i W_{\ov 0}}
+\sum_{j=1}^{n}\rho_{00}\ov{\rho_{0j}}\tilde \alpha_{W_i W_{\ov 0}} +\sum_{i,j=1}^{n}\rho_{0i}\ov \rho_{0j}\tilde \alpha_{W_i W_{\ov j}} \\
\leq {} & |\rho_{00}|^{2}\tilde \alpha_{W_0 W_{\ov 0}}+C\sqrt{\mu_{0}}\sum_{i=1}^{n}|\rho_{0i}|+C \\
= {} & |\rho_{00}|^{2}\left(\mu_{0}+\sum_{i=1}^{n}\mu_{i}-\sum_{i=1}^n \tilde \alpha_{W_i W_{\ov{j}}}\right)+C\sqrt{\mu_{0}}\sum_{i=1}^{n}|\rho_{0i}|+C \\
\leq {} & |\rho_{00}|^{2}\mu_{0}+C\sqrt{\mu_{0}}\sum_{i=1}^{n}|\rho_{0i}|+C.
\end{split}
\]
Applying $|\rho_{00}|^{2}=1-\sum_{i=1}^{n}|\rho_{0i}|^{2}$, we get
\[
\mu_{0} \leq \left(1-\sum_{i=1}^{n}|\rho_{0i}|^{2}\right)\mu_{0}+C\sqrt{\mu_{0}}\sum_{i=1}^{n}|\rho_{0i}|+C,
\]
It then follows that
\[
\mu_{0}\sum_{i=1}^{n}|\rho_{0i}|^{2} \leq C\sqrt{\mu_{0}}\sum_{i=1}^{n}|\rho_{0i}|+C,
\]
which implies
\begin{equation}\label{Estimate 2 eqn 4}
\sum_{i=1}^{n}|\rho_{0i}|^2 \leq \frac{C}{\mu_{0}}.
\end{equation}

Then (2) follows.
\end{proof}

\begin{lemma}\label{Calculation}
At $x_{0}$, we have
\begin{equation}\label{Calculation 1}
L(|\nabla^{X}\vp|^{2}) \geq \sum_{i=0}^{n}\sum_{k=1}^{n}F^{i\ov{i}}\left(|\de_{i}W_{k}(\vp)|^{2}+|\de_{i}W_{\ov{k}}(\vp)|^{2}\right)-C
\end{equation}
and
\begin{equation}\label{Calculation 2}
L\left(g_{VV}^{-1}\vp_{VV}\right)
\geq \sum_{i,j=0}^{n}\frac{\mu_{i}+\mu_{j}}{(1+\mu_{i}^{2})(1+\mu_{j}^{2})}|V(\ti{\alpha}_{i\ov{j}})|^{2}-C\vp_{VV}.
\end{equation}
\end{lemma}

\begin{proof}
First, we have
\[
|\nabla^{X}\vp|^{2} = \sum_{k,l=0}^{n}g^{W_{k}W_{\ov{l}}}W_{k}(\vp)W_{\ov{l}}(\vp).
\]
For (\ref{Calculation 1}), we compute
\begin{equation}\label{Calculation 1 eqn 1}
\begin{split}
L(|\nabla^{X}\vp|^{2})
= {} & \sum_{i=0}^{n}\sum_{k=1}^{n}F^{i\ov{i}}\left(|\de_{i}W_{k}(\vp)|^{2}+|\de_{i}W_{\ov{k}}(\vp)|^{2}\right) \\
& +\sum_{i=0}^{n}F^{i\ov{i}}\de_{i}\de_{\ov{i}}(g^{W_{k}W_{\ov{l}}})W_{k}(\vp)W_{\ov{l}}(\vp) \\
& +2{\rm Re}\left(\sum_{i=0}^{n}F^{i\ov{i}}\de_{i}\de_{\ov{i}}W_{k}(\vp)W_{\ov{k}}(\vp)\right).
\end{split}
\end{equation}
To deal with the last term, we apply $W_{k}$ to the equation (\ref{SLPE}), and obtain
\[
\sum_{i=0}^{n}F^{i\ov{i}}W_{k}(\ti{\alpha}_{i\ov{i}}) = W_{k}(h),
\]
which implies
\[
\begin{split}
\sum_{i=0}^{n}F^{i\ov{i}}\de_{i}\de_{\ov{i}}W_{k}(\vp)
=  W_{k}(h)-\sum_{i=0}^{n}F^{i\ov{i}}W_{k}((\pi^{*}\alpha)_{i\ov{i}}) \geq -C,
\end{split}
\]
where we used $F^{i\ov{i}}\leq 1$ in the last inequality. Substituting this into (\ref{Calculation 1 eqn 1}) and the uniform estimate (\ref{Spatial estimate}) for the spatial second order derivatives, we have
\[
\begin{split}
L(|\nabla^{X}\vp|^{2})
= {} & \sum_{i=0}^{n}\sum_{k=1}^{n}F^{i\ov{i}}\left(|\de_{i}W_{k}(\vp)|^{2}+|\de_{i}W_{\ov{k}}(\vp)|^{2}\right)
-C|\nabla^{X}\vp|^{2}-C|\nabla^{X}\vp| \\
\geq {} & \sum_{i=0}^{n}\sum_{k=1}^{n}F^{i\ov{i}}\left(|\de_{i}W_{k}(\vp)|^{2}+|\de_{i}W_{\ov{k}}(\vp)|^{2}\right)-C,
\end{split}
\]
as required.

For (\ref{Calculation 2}), we compute
\begin{equation}\label{Calculation 2 eqn 1}
\begin{split}
& L(g_{VV}^{-1}\vp_{VV}) \\[1mm]
= {} & \sum_{i=0}^{n}F^{i\ov{i}}\de_{i}\de_{\ov{i}}(\vp_{VV})+\sum_{i=0}^{n}F^{i\ov{i}}\de_{i}\de_{\ov{i}}(g_{VV}^{-1})\vp_{VV} \\
\geq {} & \sum_{i=0}^{n}F^{i\ov{i}}\de_{i}\de_{\ov{i}}VV(\vp)+\sum_{i=0}^{n}F^{i\ov{i}}\de_{i}\de_{\ov{i}}(\nabla_{V}V)(\vp)
-C\sum_{i=0}^{n}F^{i\ov{i}} \\
\geq {} & \sum_{i=0}^{n}F^{i\ov{i}}VV\de_{i}\de_{\ov{i}}(\vp)+\sum_{i=0}^{n}F^{i\ov{i}}\de_{i}\de_{\ov{i}}(\nabla_{V}V)(\vp)-C,
\end{split}
\end{equation}
where we used $F^{i\ov{i}}\leq 1$ in the last inequality.

For the first term on the right hand of (\ref{Calculation 2 eqn 1}), we apply $VV$ to the equation (\ref{SLPE}) and obtain
\[
\sum_{i=0}^{n}F^{i\ov{i}}VV(\ti{\alpha}_{i\ov{i}})
-\sum_{i,j=0}^{n}\frac{\mu_{i}+\mu_{j}}{(1+\mu_{i}^{2})(1+\mu_{j}^{2})}|V(\ti{\alpha}_{i\ov{j}})|^{2}
= VV(h),
\]
which implies
\begin{equation}\label{Calculation 2 eqn 2}
\begin{split}
& \sum_{i=0}^{n}F^{i\ov{i}}VV\de_{i}\de_{\ov{i}}(\vp) \\
= {} & \sum_{i,j=0}^{n}\frac{\mu_{i}+\mu_{j}}{(1+\mu_{i}^{2})(1+\mu_{j}^{2})}|V(\ti{\alpha}_{i\ov{j}})|^{2}
-\sum_{i=0}^{n}F^{i\ov{i}}VV((\pi^{*}\alpha)_{i\ov{i}})+VV(h) \\
\geq {} & \sum_{i,j=0}^{n}\frac{\mu_{i}+\mu_{j}}{(1+\mu_{i}^{2})(1+\mu_{j}^{2})}|V(\ti{\alpha}_{i\ov{j}})|^{2}-C.
\end{split}
\end{equation}

For the second term on the right hand of (\ref{Calculation 2 eqn 1}), we have
\begin{equation}\label{Calculation 2 eqn 3}
\begin{split}
\left|\sum_{i=0}^{n}F^{i\ov{i}}\de_{i}\de_{\ov{i}}(\nabla_{V}V)(\vp)\right|
\leq {} & \sum_{i=0}^{n}\left|\de_{i}\de_{\ov{i}}(\nabla_{V}V)(\vp)\right| \\
= {} & \sum_{i,k,l=0}^{n}\left|\rho_{ik}\ov{\rho_{il}}W_{k}W_{\ov{l}}(\nabla_{V}V)(\vp)\right|.
\end{split}
\end{equation}
We claim
\begin{equation}\label{Calculation 2 claim}
\sum_{i,k,l=0}^{n}\left|\rho_{ik}\ov{\rho_{il}}W_{k}W_{\ov{l}}(\nabla_{V}V)(\vp)\right| \leq C\vp_{VV}.
\end{equation}
Substituting (\ref{Calculation 2 eqn 2}), (\ref{Calculation 2 eqn 3}) and (\ref{Calculation 2 claim}) into (\ref{Calculation 2 eqn 1}), we obtain (\ref{Calculation 2}).

Now we prove the claim (\ref{Calculation 2 claim}). It suffices to prove each term can be controlled by $\vp_{VV}$. There are four cases:

\bigskip
\noindent
{\bf Case 1:} \ $k=l=0$.
\bigskip

In this case, since $W_{0}=\de_{t}$ is time vector field and $\nabla_{V}V$ is a spatial vector field, then we have
\[
[W_{0},\nabla_{V}V] = 0, \ [W_{\ov{0}},\nabla_{V}V] = 0.
\]
Combining this with $(\nabla_{V}V)(x_{0})=0$, we obtain
\[
W_{0}W_{\ov{0}}(\nabla_{V}V)(\vp) = (\nabla_{V}V)W_{0}W_{\ov{0}}(\vp) = 0,
\]
which implies
\[
\left|\rho_{i0}\ov{\rho_{i0}}W_{0}W_{\ov{0}}(\nabla_{V}V)(\vp)\right| = 0.
\]

\bigskip
\noindent
{\bf Case 2:} \ $k=0, l\neq0$.
\bigskip

In this case, we compute
\[
W_{0}W_{\ov{l}}(\nabla_{V}V)(\vp)
= W_{\ov{l}}(\nabla_{V}V)W_0(\vp)
= [W_{\ov{l}},\nabla_{V}V]W_{0}(\vp),
\]
where we used that $\nabla_{V}V(x_0)=0$.  Since $W_{\ov{l}}$ and $\nabla_{V}V$ are spatial vector field, the Lie bracket $[W_{\ov{l}},\nabla_{V}V]$ is still a spatial vector field. We assume
\[
[W_{\ov{l}},\nabla_{V}V] = \sum_{p=1}^{n}\left(s_{\ov{l}}^{p}W_{p}+s_{\ov{l}}^{\ov{p}}W_{\ov{p}}\right).
\]
By Lemma \ref{Estimate 2}, we obtain
\[
\begin{split}
\left|W_{0}W_{\ov{l}}(\nabla_{V}V)(\vp)\right|
\leq {} & \sum_{p=1}^{n}\left(|s_{\ov{l}}^{p}W_{p}W_{0}(\vp)|+|s_{\ov{l}}^{\ov{p}}W_{\ov{p}}W_{0}(\vp)|\right) \\
\leq {} & C\sum_{p=1}^{n}\left(|W_{p}W_{0}(\vp)|+|W_{\ov{p}}W_{0}(\vp)|\right) \\[2mm]
\leq {} & C\sqrt{\mu_{0}}.
\end{split}
\]
Combining this with Lemma \ref{Estimate 1}, we have
\[
\left|\rho_{i0}\ov{\rho_{il}}W_{0}W_{\ov{l}}(\nabla_{V}V)(\vp)\right|
\leq |\rho_{i0}|\left|W_{0}W_{\ov{l}}(\nabla_{V}V)(\vp)\right|
\leq C.
\]

\bigskip
\noindent
{\bf Case 3:} \ $k\neq 0, l=0$.
\bigskip

This case is similar to Case 2.

\bigskip
\noindent
{\bf Case 4:} \ $k\neq 0, l\neq0$.
\bigskip

In this case, we compute
\[
\begin{split}
& W_{k}W_{\ov{l}}(\nabla_{V}V)(\vp) \\
= {} & W_{k}[W_{l},\nabla_{V}V](\vp)+W_{k}(\nabla_{V}V)W_{\ov{l}}(\vp) \\
= {} & [W_{l},\nabla_{V}V]W_{k}(\vp)+[[W_{l},\nabla_{V}V],W_{k}](\vp)+[W_{k},\nabla_{V}V]W_{\ov{l}}(\vp).
\end{split}
\]
Since $W_{k}$, $W_{\ov{l}}$, $[W_{l},\nabla_{V}V]$ and $[W_{k},\nabla_{V}V]$ are spatial vector fields, using uniform estimate (\ref{Spatial estimate}) for the spatial second order derivatives and the definition of $V$, we see that
\[
\left|W_{k}W_{\ov{l}}(\nabla_{V}V)(\vp)\right| \leq C\vp_{VV}.
\]
Thus,
\[
\left|\rho_{ik}\ov{\rho_{il}}W_{k}W_{\ov{l}}(\nabla_{V}V)(\vp)\right| \leq C\vp_{VV}.
\]
\end{proof}

\begin{lemma}\label{Third order terms}
At $x_{0}$, we have
\begin{equation}\label{Third order terms eqn 1}
\begin{split}
\sum_{i,j=0}^{n}\frac{(\mu_{i}+\mu_{j})|V(\ti{\alpha}_{i\ov{j}})|^{2}}{(1+\mu_{i}^{2})(1+\mu_{j}^{2})\vp_{VV}^{2}}
& - \sum_{i=0}^{n}\frac{F^{i\ov{i}}|\de_{i}(\vp_{VV})|^{2}}{\vp_{VV}^{2}} \\
& \geq  (f')^{2}\sum_{i=0}^{n}F^{i\ov{i}}|\de_{i}(|\nabla^{X}\vp|^{2})|^{2}-C.
\end{split}
\end{equation}
\end{lemma}

\begin{proof}
First, at $x_{0}$, we have
\[
0 = \de_{i}\hat{Q} = \frac{\de_{i}(\vp_{VV})}{\vp_{VV}}+f'\de_{i}(|\nabla^{X}\vp|^{2}),
\]
which implies
\[
\sum_{i=0}^{n}\frac{F^{i\ov{i}}|\de_{i}(\vp_{VV})|^{2}}{\vp_{VV}^{2}}
= (f')^{2}\sum_{i=0}^{n}F^{i\ov{i}}|\de_{i}(|\nabla^{X}\vp|^{2})|^{2}.
\]

Next we use the idea of \cite{CY18} to deal with the first term on the left hand side of (\ref{Third order terms eqn 1}). If $\mu_{n}\geq0$, then
\[
\sum_{i,j=0}^{n}\frac{(\mu_{i}+\mu_{j})|V(\ti{\alpha}_{i\ov{j}})|^{2}}{(1+\mu_{i}^{2})(1+\mu_{j}^{2})\vp_{VV}^{2}} \geq 0,
\]
which implies (\ref{Third order terms eqn 1}). Hence we assume that $\mu_{n}<0$. By Lemma \ref{Properties of Lagrangian} (1), we have
\[
\begin{split}
\sum_{i,j=0}^{n}\frac{(\mu_{i}+\mu_{j})|V(\ti{\alpha}_{i\ov{j}})|^{2}}{(1+\mu_{i}^{2})(1+\mu_{j}^{2})\vp_{VV}^{2}}
\geq {} & \sum_{k=0}^{n}\frac{2\mu_{k}|V(\ti{\alpha}_{k\ov{k}})|^{2}}{(1+\mu_{k}^{2})^{2}\vp_{VV}^{2}} \\
= {} & \sum_{k=0}^{n-1}\frac{2\mu_{k}|V(\ti{\alpha}_{k\ov{k}})|^{2}}{(1+\mu_{k}^{2})^{2}\vp_{VV}^{2}}
-\frac{2|\mu_{n}| |V(\ti{\alpha}_{n\ov{n}})|^{2}}{(1+\mu_{n}^{2})^{2}\vp_{VV}}.
\end{split}
\]
Since we assume $\vp_{VV}\geq1$, it suffices to prove
\[
\frac{|\mu_{n}| |V(\ti{\alpha}_{n\ov{n}})|^{2}}{(1+\mu_{n}^{2})^{2}}
\leq \sum_{k=0}^{n-1}\frac{\mu_{k}|V(\ti{\alpha}_{k\ov{k}})|^{2}}{(1+\mu_{k}^{2})^{2}}.
\]
Applying $V$ to the (eqn), we obtain
\[
F^{i\ov{i}}V(\ti{\alpha}_{i\ov{i}}) = V(h),
\]
which implies
\[
\frac{V(\ti{\alpha}_{n\ov{n}})}{1+\mu_{n}^{2}}
= V(h)-\sum_{k=0}^{n-1}\frac{V(\ti{\alpha}_{k\ov{k}})}{1+\mu_{k}^{2}}.
\]
Using the Cauchy-Schwarz inequality twice, we see that
\[
\begin{split}
\frac{|\mu_{n}| |V(\ti{\alpha}_{n\ov{n}})|}{(1+\mu_{n}^{2})^{2}}
\leq {} & (1+\delta)|\mu_{n}|\left|\sum_{k=0}^{n-1}\frac{V(\ti{\alpha}_{k\ov{k}})}{1+\mu_{k}^{2}}\right|
+\left(1+\frac{1}{\delta}\right)|\mu_{n}||V(h)| \\
\leq {} & (1+\delta)|\mu_{n}|\left(\sum_{k=0}^{n-1}\frac{1}{\mu_{k}}\right)\left( \sum_{k=0}^{n-1}\frac{\mu_{k}|V(\ti{\alpha}_{k\ov{k}})|^{2}}{(1+\mu_{k}^{2})^{2}}\right) +\frac{C}{\delta},
\end{split}
\]
where $\delta$ is a constant to be determined later.

On the other hand, by Lemma \ref{Properties of Lagrangian} (2), we have
\[
\sum_{k=0}^{n-1}\frac{1}{\mu_{k}} \leq \frac{n}{\tan(\eta)}.
\]
Applying Lemma \ref{Properties of Lagrangian} (2) again, we obtain
\[
\sum_{k=0}^{n-1}\frac{1}{\mu_{k}} \leq \sum_{k=0}^{n}\frac{1}{\mu_{k}}-\frac{1}{\mu_{n}} \leq -\tan(\eta)+\frac{1}{|\mu_{n}|}.
\]
Hence,
\[
|\mu_{n}|\left(\sum_{k=0}^{n-1}\frac{1}{\mu_{k}}\right)
\leq \min\left(\frac{n|\mu_{n}|}{\tan(\eta)},1-|\mu_{n}|\tan(\eta)\right)
\leq \frac{n}{n+\tan^{2}(\eta)}.
\]
Now we choose $\delta=\frac{\tan^{2}(\eta)}{n}$, it then follows that
\[
\begin{split}
\frac{|\mu_{n}| |V(\ti{\alpha}_{n\ov{n}})|^{2}}{(1+\mu_{n}^{2})^{2}}
\leq {} & \frac{n(1+\delta)}{n+\tan^{2}(\eta)}\sum_{k=0}^{n-1}\frac{\mu_{k}|V(\ti{\alpha}_{k\ov{k}})|^{2}}{(1+\mu_{k}^{2})^{2}}+\frac{C}{\delta} \\
\leq {} & \sum_{k=0}^{n-1}\frac{\mu_{k}|V(\ti{\alpha}_{k\ov{k}})|^{2}}{(1+\mu_{k}^{2})^{2}}+C,
\end{split}
\]
as desired.
\end{proof}

Now we complete the proof of Theorem \ref{C11 estimate}. By the maximum principle, at $x_{0}$, we have
\[
\begin{split}
0 \geq L\hat{Q}
= {} & \frac{L(g_{VV}^{-1}\vp_{VV})}{\vp_{VV}}-\sum_{i=0}^{n}\frac{F^{i\ov{i}}|\de_{i}(\vp_{VV})|^{2}}{\lambda_{1}^{2}}
 +f'L(|\nabla^{X}\vp|^{2})+f''\sum_{i=0}^{n}F^{i\ov{i}}|\de_{i}(|\nabla^{X}\vp|^{2})|^{2} \\
\geq {} & \sum_{i,j=0}^{n}\frac{\mu_{i}+\mu_{j}}{(1+\mu_{i}^{2})(1+\mu_{j}^{2})}|V(\ti{\alpha}_{i\ov{j}})|^{2}
+(f''-(f')^{2})\sum_{i=0}^{n}F^{i\ov{i}}|\de_{i}(|\nabla^{X}\vp|^{2})|^{2} \\
& +f'\sum_{i=0}^{n}\sum_{k=1}^{n}F^{i\ov{i}}\left(|\de_{i}W_{k}(\vp)|^{2}+|\de_{i}W_{\ov{k}}(\vp)|^{2}\right)
-Cf'-C
\end{split}
\]
Using Lemma \ref{Third order terms}, $f''-(f')^{2}=0$ and $f'\leq C$, we obtain
\begin{equation}\label{C11 estimate eqn 1}
\sum_{i=0}^{n}\sum_{k=1}^{n}F^{i\ov{i}}\left(|\de_{i}W_{k}(\vp)|^{2}+|\de_{i}W_{\ov{k}}(\vp)|^{2}\right) \leq C.
\end{equation}
We claim
\begin{equation}\label{C11 estimate claim}
\sum_{i,j=1}^{n}|W_{i}W_{j}(\vp)| \leq C.
\end{equation}
Given this claim, by the definition of $V$ and Lemma, we have
\[
\vp_{VV} \leq C\sum_{i,j=1}^{n}\left(|W_{i}W_{j}(\vp)|+|W_{i}W_{\ov{j}}(\vp)|\right) \leq C,
\]
as required.

It suffices to prove the claim (\ref{C11 estimate claim}). Recalling Lemma \ref{Estimate 2}, there exists a constant $C_{0}$ such that
\[
|\eta_{0i}|+|\eta^{0i}| \leq \frac{C_{0}}{\sqrt{\mu_{0}}}, \ \text{for $1\leq i \leq n$}.
\]
The proof of the claim (\ref{C11 estimate claim}) splits into two cases:

\bigskip
\noindent
{\bf Case 1:} \ $\mu_{0}\leq 4C_{0}^{2}$.
\bigskip

In this case, we have $F^{i\ov{i}} \geq C^{-1}$ for $0\leq i\leq n$. From (\ref{C11 estimate eqn 1}), we obtain
\[
\sum_{i=0}^{n}\sum_{k=1}^{n}|\de_{i}W_{k}(\vp)|\leq C.
\]
Therefore,
\[
\sum_{i,k=1}^{n}|W_{i}W_{k}(\vp)|
= \sum_{i,k=1}^{n}\left|\sum_{p=0}^{n}\rho^{ip}\de_{p}W_{k}(\vp)\right|
\leq C\sum_{i=0}^{n}\sum_{k=1}^{n}|\de_{i}W_{k}(\vp)|\leq C.
\]

\bigskip
\noindent
{\bf Case 2:} \ $\mu_{0}\geq 4C_{0}^{2}$.
\bigskip

In this case, we have $F^{i\ov{i}} \geq C^{-1}$ for $1\leq i\leq n$. From (\ref{C11 estimate eqn 1}), we obtain
\[
\sum_{i=1}^{n}\sum_{k=1}^{n}|\de_{i}W_{k}(\vp)|\leq C.
\]
Combining (\ref{C11 estimate eqn 1}) and Lemma \ref{Estimate 2}, we compute
\begin{equation}\label{C11 estimate case 2 eqn 1}
\begin{split}
\sum_{i,k=1}^{n}|W_{i}W_{k}(\vp)|
= {} & \sum_{i,k=1}^{n}\left|\sum_{p=0}^{n}\rho^{ip}\de_{p}W_{k}(\vp)\right| \\
\leq {} & \sum_{i,k,p=1}^{n}|\rho^{ip}\de_{i}W_{k}(\vp)|+\sum_{i,k=1}^{n}|\rho^{i0}\de_{0}W_{k}(\vp)| \\
\leq {} & C+\frac{C_{0}}{\sqrt{\mu_{0}}}\sum_{k=1}^{n}|\de_{0}W_{k}(\vp)|.
\end{split}
\end{equation}
Using Lemma \ref{Estimate 1}, we have
\begin{equation}\label{C11 estimate case 2 eqn 2}
\begin{split}
\sum_{k=1}^{n}|\de_{0}W_{k}(\vp)|
\leq {} & \sum_{k=1}^{n}\sum_{p=0}^{n}|\rho_{0p}W_{p}W_{k}(\vp)| \\
\leq {} & \sum_{k,p=1}^{n}|W_{p}W_{k}(\vp)|+\sum_{k=1}^{n}|W_{0}W_{k}(\vp)| \\
\leq {} & \sum_{i,k=1}^{n}|W_{i}W_{k}(\vp)|+C\sqrt{\mu_{0}}.
\end{split}
\end{equation}
Substituting (\ref{C11 estimate case 2 eqn 2}) into (\ref{C11 estimate case 2 eqn 1}), we conclude that
\[
\sum_{i,k=1}^{n}|W_{i}W_{k}(\vp)|
\leq \frac{C_{0}}{\sqrt{\mu_{0}}}\sum_{i,k=1}^{n}|W_{i}W_{k}(\vp)|+C.
\]
Since $\mu_{0}\geq 4C_{0}^{2}$, we obtain
\[
\sum_{i,j=1}^{n}|W_{i}W_{j}(\vp)| \leq C.
\]
\end{proof}

\subsection{Examples}
In this subsection we construct some examples which show that the weak geodesics in $\mathcal{H}$ are not $C^{2}$ in general.

\subsubsection{Manifolds of dimension $n=1$}
Let $(X,\omega)$ be a compact $1$-dimensional K\"ahler manifold and $\alpha=\omega$. Recalling the definition of $\hat{\theta}$, it is clear that $\hat{\theta}=\frac{\pi}{4}$, and hence $[\omega]$ has hypercritical phase. We consider the space of K\"ahler potentials with respect to $2\omega$:
\[
\mathcal{H}_{{\rm PSH}} :=
\{\phi\in C^{\infty}(X)~|~2\omega+\ddbar\phi>0\}.
\]
Since
\[
{\rm Re}\left(e^{-\sqrt{-1}\hat{\theta}}(\omega+\sqrt{-1}\omega_{\phi})\right)
= \frac{\sqrt{2}}{2}(2\omega+\ddbar\phi),
\]
we have
\[
\mathcal{H}=\mathcal{H}_{{\rm PSH}}.
\]

In this case one can easily check that the Riemannian structure on $\mathcal{H}$ studied here agrees exactly with the Donaldson-Mabuchi-Semmes Riemannian structure.  In particular, they have the same geodesics.  Let $\Phi$ denote this common geodesic.

Consider the following concrete example. Let $(T,\omega_{T})$ be the standard $1$-dimensional complex torus and $f$ be the holomorphic isometry induced by $z\mapsto-z$ in $\mathbb{C}$. By \cite[Theorem 1.1]{DL12} (see also \cite{LV13,Darvas14}), there exist $\phi_{0},\phi_{1}\in\mathcal{H}_{{\rm PSH}}$ such that the weak geodesic $\Psi$ joining them is not $C^{2}$. Combining this with the above argument, we obtain $\Phi\notin C^{2}(T\times\mathcal{A})$.

\subsubsection{Manifolds of dimension $n>1$} Let $(M,\omega_{M})$ be a compact $(n-1)$-dimensional K\"ahler manifold. We consider the product manifold $T\times M$, where $T$ is the torus from before, and denote the projection from $T\times M$ to $T$, $M$ by $p_{1}$, $p_{2}$ respectively. Then
\[
(X,\omega) := (T\times M,p_{1}^{*}\omega_{T}+p_{2}^{*}\omega_{M})
\]
is a compact $n$-dimensional K\"ahler manifold. We define
\[
\alpha = p_{1}^{*}\omega_{T}+Ap_{2}^{*}\omega_{M},
\]
where $A$ is a positive constant to be determined. It is clear that
\[
\hat{\theta} = \frac{\pi}{4}+(n-1)\arctan A.
\]
Choosing $A$ sufficiently large, $\hat{\theta}$ satisfies the ``hypercritical phase" condition:
\[
\hat{\theta} \in \left((n-1)\frac{\pi}{2},n\frac{\pi}{2}\right).
\]
For convenience, we use $\phi_{0}$, $\phi_{1}$, $\Phi$ denote the same functions as above. It is not hard to check that $p_{1}^{*}\phi_{0},p_{1}^{*}\phi_{1}\in\mathcal{H}$ and $p^{*}\Phi$ is the unique weak geodesic joining $p_{1}^{*}\phi_{0},p_{1}^{*}\phi_{1}$, where $p$ is the projection $X\times\mathcal{A}\rightarrow T\times\mathcal{A}$. Since $\Phi\notin C^{2}(T\times\mathcal{A})$, we have $p^{*}\Phi\notin C^{2}(X\times\mathcal{A})$.

\end{document}